\documentclass[preprint,12pt]{elsarticle}

\usepackage{graphicx}
\usepackage{epstopdf}
\usepackage{subcaption} 
\graphicspath{ {./Figures/} }

\usepackage{comment}
\usepackage{mathrsfs}
\usepackage{amsfonts}
\usepackage{amssymb}
\usepackage{bm}
\usepackage{xcolor}
\usepackage{amsmath}
\usepackage{mathtools}
\usepackage{amsthm}
\usepackage[hidelinks]{hyperref}
\newcounter{casenum}

\newcommand{\dbar}[1]{\bar{\bar{#1}}}
\newcommand{\dover}[1]{\overline{\overline{#1}}} 

\usepackage{bm}
\usepackage{xspace}

\newcommand{\Tr}{\ensuremath{^{\mr{T}}}}

\newcommand{\mr}[1]{\ensuremath{\mathrm{#1}}}
\newcommand{\mat}[1]{\ensuremath{\mathsf{#1}}}
\newcommand{\fnc}[1]{\ensuremath{\mathit{#1}}}
\newcommand{\bfnc}[1]{\ensuremath{\bm{\mathit{#1}}}}

\newcommand{\xm}[0]{\ensuremath{x_{m}}}

\newcommand{\xil}[0]{\ensuremath{\xi_{l}}}

\newcommand{\Ohat}[0]{\ensuremath{\hat{\Omega}}}

\newcommand{\U}[0]{\ensuremath{\bfnc{U}}}
\newcommand{\Fxm}[0]{\ensuremath{\bfnc{F}}_{\xm}}

\newcommand{\Vone}[0]{\ensuremath{\fnc{V}_{1}}}
\newcommand{\Vtwo}[0]{\ensuremath{\fnc{V}_{2}}}
\newcommand{\Vthree}[0]{\ensuremath{\fnc{V}_{3}}}
\newcommand{\Vm}[0]{\ensuremath{\fnc{V}_{m}}}

\newcommand{\E}[0]{\ensuremath{\fnc{E}}}

\newcommand{\Cij}[2]{\ensuremath{\mat{C}_{#1,#2}}}

\newcommand{\Fxmv}[0]{\ensuremath{\bm{\fnc{F}}_{\xm}^{(v)}}}

\theoremstyle{plain}
\newtheorem{theorem}{Theorem}
\newtheorem{corollary}{Corollary}[theorem]
\newtheorem{lemma}[theorem]{Lemma}
\theoremstyle{definition}

\newtheorem{remark}{Remark}

\begin{document}

\begin{frontmatter}



\title{First-order positivity-preserving entropy stable spectral collocation scheme for the 3-D compressible Navier-Stokes equations}


\author{Johnathon Upperman and Nail K. Yamaleev\footnote{Corresponding author. Department of Mathematics and Statistics, Tel.: +1 757 683 3423. {\it E-mail address:} nyamalee@odu.edu}}

\address{Old Dominion University, Norfolk, VA 23529, USA}

\begin{abstract}
In this paper, we extend the positivity--preserving, entropy stable first--order finite volume-type  scheme developed for the one-dimensional compressible Navier-Stokes equations in \cite{UY_1Dlow} to three spatial dimensions. The new first-order scheme is provably entropy stable, design--order accurate for smooth solutions, and guarantees the pointwise positivity of thermodynamic variables for 3-D compressible viscous flows. 
Similar to the 1-D counterpart, the proposed scheme for the 3-D Navier-Stokes equations is discretized on Legendre-Gauss-Lobatto grids used for high-order  spectral collocation methods. The positivity of density is achieved by adding an artificial dissipation in the form of the first-order Brenner-Navier-Stokes diffusion operator. 
Another distinctive feature of the proposed scheme is that the Navier--Stokes viscous terms are discretized by high--order spectral collocation summation-by-parts operators.
To eliminate time step stiffness caused by the high-order approximation of the viscous terms, the velocity and temperature limiters developed for the 1-D compressible Navier-Stokes equations in \cite{UY_1Dlow} are generalized to three spatial dimensions. These limiters bound the magnitude of velocity and temperature gradients and preserve the entropy stability and positivity properties of the baseline scheme.
Numerical results are presented to demonstrate design-order accuracy and positivity-preserving properties of the new first-order scheme for 2-D and 3-D inviscid and viscous flows with strong shocks and contact discontinuities.
\end{abstract}

\begin{keyword}
summation-by-parts (SBP) operators, entropy stability, positivity-preserving schemes, Brenner regularization, artificial dissipation, the Navier-Stokes equations.
\end{keyword}

\end{frontmatter}


\section{Introduction}
We have recently developed a novel first--order entropy stable finite volume scheme that provides pointwise positivity of thermodynamic variables for the 1-D compressible Navier-Stokes equations on Legendre-Gauss-Lobatto (LGL) grids used for high-order spectral collocation methods \cite{UY_1Dlow}. 
The positivity preservation and entropy stability properties are achieved by introducing the first-order artificial dissipation operator that mimics the corresponding diffusion operator of the Brenner-Navier-Stokes equations \cite{Brenner1}. It has been proven that the proposed first-order scheme is conservative, entropy stable, and positivity preserving for the 1-D compressible Euler and Navier-Stokes equations. This positivity-preserving methodology has recently been extended to entropy stable spectral collocation schemes of arbitrary order of accuracy for the 1-D Navier-Stokes equations in the companion paper \cite{UY_1Dhigh}. Herein, we generalize and extend this 1-D positivity-preserving entropy stable finite volume scheme developed in \cite{UY_1Dlow} to the three-dimensional compressible Navier-Stokes equations on fully unstructured static hexahedral grids. 

There are very few papers available in the literature on positivity-preserving methods for the compressible Navier-Stokes equations especially in three spatial dimensions. A first-order positivity-preserving finite difference scheme for the 3-D compressible Navier-Stokes equations on Cartesian uniform grids has been developed in \cite{Svard}. The positivity proof in this paper is based on some memetic properties of the 1st-order finite difference operators on uniform grids, which are not available for curvilinear or unstructured grids. 
In \cite{GHKL}, an implicit first-order positivity-preserving scheme is constructed for the compressible Navier-Stokes equations on staggered grids. This scheme is unconditionally stable and solves the internal energy equation instead of the equation for conservation of total energy. 
Another positivity-preserving scheme for the compressible Navier-Stokes equations has been proposed in \cite{GMPT}. This approach relies on the invariant domain preserving approximation of the Euler equations and the Strang's operator splitting technique that is at most 2nd-order accurate.  Note that no entropy stability proof is currently available for the numerical schemes developed in \cite{GHKL, GMPT}. 
Recently, a high-order positivity-preserving discontinuous Galerkin (DG) scheme for the compressible Navier-Stokes equations is presented in \cite{Zhang}. This explicit in time method provides only so-called weak positivity of density and pressure and imposes very severe constraints on the time step. Note that the actual time step constraint is much stiffer, because the lower bound on the artificial viscosity coefficient, which is required for density and pressure positivity, may grow dramatically, as the velocity gradients increase. 

The main objective of the present paper is to 
construct a first-order positivity-preserving entropy stable scheme defined on LGL grids which are used for high-order spectral collocation methods (e.g., \cite{CFNF, CFNPSY}). To provide the positivity of thermodynamic variables, we construct new first-order artificial dissipation operators that are based on the Brenner regularization of the compressible Navier-Stokes equations \cite{UY_1Dlow, UY}. In contrast to the existing positivity-preserving schemes that rely on monotonicity properties of the Rusanov-type dissipation, the proposed method minimizes the amount of artificial dissipation required for pointwise positivity of density and temperature and uses novel velocity and temperature limiters to eliminate the time step stiffness for viscous flows with strong shock waves and contact discontinuities.

\section{3-D regularized  Navier-Stokes equations}
\label{BNS}
We consider the 3-D compressible Navier-Stokes equations in curvilinear coordinates $(\xi_1,\xi_2,\xi_3)$, which are  written in conservation law form as follows:
\begin{equation}
\label{eq:NS_Curvilinear}
\begin{split}
&\frac{\partial J\U}{\partial t} + \sum\limits_{m, l=1}^{3}\frac{\partial}{\partial \xi_l}\left(\bfnc{F}_{\xil}-\bfnc{F}_{\xil}^{(v)}\right)= 0,\\
&\bfnc{F}_{\xil}\equiv \sum\limits_{m=1}^{3}\fnc{J}\frac{\partial \xi_l}{\partial \xm}\Fxm,\quad
\bfnc{F}_{\xil}^{(v)}\equiv\sum\limits_{m=1}^{3}\fnc{J}\frac{\partial \xi_l}{\partial \xm}\Fxm^{(v)},
\end{split}
\end{equation}
where $\U = \left[\rho,\rho\Vone,\rho\Vtwo,\rho\Vthree,\rho\E\right]\Tr$ is a vector of the conservative variables  and
$\Fxm$, and $\Fxmv$ are the inviscid and viscous fluxes associated with the Cartesian coordinates $(x_1, x_2, x_3)$, which are given by
\begin{equation*}
\Fxm = \left[\rho\Vm,\rho\Vm\Vone+\delta_{m,1}\fnc{P},\rho\Vm\Vtwo+\delta_{m,2}\fnc{P},\rho\Vm\Vthree+\delta_{m,3}\fnc{P},\rho\Vm\fnc{H}\right]\Tr, 
\end{equation*}
\begin{equation}\label{eq:Fv}
\Fxmv=\left[0,\tau_{1,m},\tau_{2,m},\tau_{3,m},\sum\limits_{i=1}^{3}\tau_{i,m}\fnc{V}_{i}-\kappa\frac{\partial \fnc{T}}{\partial\xm}\right]\Tr.
\end{equation} 
The viscous stresses in Eq.~(\ref{eq:Fv}) are defined as follows:
\begin{equation}\label{eq:tau}
\tau_{i,j} = \mu\left(\frac{\partial\fnc{V}_{i}}{\partial x_{j}}+\frac{\partial\fnc{V}_{j}}{\partial x_{i}}
-\delta_{i,j}\frac{2}{3}\sum\limits_{n=1}^{3}\frac{\partial\fnc{V}_{n}}{\partial x_{n}}\right),
\end{equation}
where $\mu(T)$ is the dynamic viscosity, $\kappa(T)$ is the thermal conductivity, and $\delta_{i,j}$ is the Kronecker delta.
To close the Navier-Stokes equations, the following constituent relations are used:
\begin{equation*}
\fnc{H} = c_{\fnc{P}}\fnc{T}+\frac{1}{2}\bfnc{V}\Tr\bfnc{V},\quad \fnc{P} = \rho R_g \fnc{T},\quad R_g = \frac{R_{u}}{M_{w}},
\end{equation*}
where $\fnc{T}$ is the temperature, $R_{u}$ is the universal gas constant, $M_{w}$ is the molecular weight of the gas, 
and $c_{\fnc{P}}$ is the specific heat capacity at constant pressure. 
Equation~(\ref{eq:NS_Curvilinear}) is subject to boundary conditions that are assumed to satisfy the entropy inequality. 
Note that only static curvilinear unstructured grids are considered in the present analysis.

Since the Navier-Stokes equations have no inherent mechanism for producing admissible solutions with positive thermodynamic variables, we regularize Eq.~(\ref{eq:NS_Curvilinear}) by adding artificial dissipation 
in the form of the diffusion operator of the Brenner-Navier-Stokes equations  introduced in \cite{Brenner1}. 
The regularized Navier-Stokes equations in curvilinear coordinates are given by
\begin{equation}
\label{eq:regNS_Curvilinear}
\frac{\partial J\U}{\partial t} 
+ 
\sum\limits_{m, l=1}^{3}
\frac{\partial}{\partial \xi_l}\left(\bfnc{F}_{\xil}
-
\bfnc{F}_{\xil}^{(v)}
-
\bm{\fnc{F}}_{\xil}^{(AD)}
\right)= 0
,\quad
\bfnc{F}^{(AD)}_{\xil}\equiv \sum\limits_{m=1}^{3}\fnc{J}\frac{\partial \xi_l}{\partial \xm}{\bm F}_{x_m}^{(AD)}
\end{equation}
$$
{\bm F}_{x_m}^{(AD)} =\left[c_{\rho}\mu^{AD}\frac{\partial \rho}{\partial x_m},\tau^{(AD)}_{1,m},\tau^{(AD)}_{2,m},\tau^{(AD)}_{3,m},\sum\limits_{i=1}^{3}\tau^{(AD)}_{i,m}\fnc{V}_{i}-c_T\mu^{AD}\frac{\partial \fnc{T}}{\partial\xm}\right]\Tr,
$$ 
$$
\tau^{(AD)}_{i,j} = \mu^{AD}\left(\frac{\partial\fnc{V}_{i}}{\partial x_{j}}+\frac{\partial\fnc{V}_{j}}{\partial x_{i}}
-\delta_{i,j}\frac{2}{3}\sum\limits_{n=1}^{3}\frac{\partial\fnc{V}_{n}}{\partial x_{n}}\right) + c_{\rho}\mu^{AD}v_i\frac{\partial \rho}{\partial x_j},
$$
where $\mu^{AD}$ is an artificial viscosity and $c_T$ and $c_\rho$ are  positive tunable coefficients.  
Equations (\ref{eq:NS_Curvilinear}) and (\ref{eq:regNS_Curvilinear}) are derived by using the following identies:
\begin{equation}
\label{eq:GCL_C}
\begin{array}{l}
\sum\limits_{ l=1}^{3}\frac{\partial}{\partial \xi_l}\left(J\frac{\partial \xi_l}{\partial \xm}\right) = 0, \quad m=1, 2, 3,
\end{array}
\end{equation}
which are called the geometric conservation laws (GCL) \cite{TLGCL}. Though, the GCL equations \eqref{eq:GCL_C} are satisfied exactly at the continuous level,
this is not the case at the discrete level \cite{TLGCL}. A discussion on how the corresponding metric 
coefficients should be discretized to satisfy the GCL equation is presented elsewhere (e.g., see \citep{TLGCL,GCLSS2019}).

Both the Navier-Stokes and regularized Navier-Stokes equations are equipped with the same scalar entropy function $\mathcal{S}= -\rho s$ and the corresponding entropy flux $\mathcal{F}=-\rho s \bm{V}$,
where $s$ is the thermodynamic entropy and $\bm{V}$ is the velocity vector. The mathematical entropy, $\mathcal{S}(\U)$, is convex and its Hessian matrix, $\frac{\partial^2 \mathcal{S}}{\partial U^2}$, is positive definite provided that $\rho>0$ and $T>0$ $\forall \bm{x} \in \Omega$, thus yielding a one-to-one mapping from the conservative to entropy variables, which are defined by $\bm{W}\Tr \equiv \frac{\partial \mathcal{S}}{\partial \U}$.
For entropy stable boundary conditions, the following entropy inequality holds for both Eqs.~(\ref{eq:NS_Curvilinear}) and (\ref{eq:regNS_Curvilinear}):
\begin{equation}\label{eq:Eineq.3}
\int_{\Ohat} \frac{\partial (J \mathcal{S})}{\partial \tau}\mr{d}\Ohat =  \frac{d}{d \tau} \int_{\Ohat} J \mathcal{S}\mr{d}\Ohat  \leq 0.
\end{equation} 
Note that the entropy inequality \eqref{eq:Eineq.3} is only a necessary condition,
which is not by itself sufficient to guarantee the convergence to a physically relevant 
weak solution of the Navier-Stokes equations.
In contrast to the conventional Navier-Stokes equations (\ref{eq:NS_Curvilinear}),
the regularized Navier-Stokes equations  \eqref{eq:regNS_Curvilinear} provide global-in-time positivity of thermodynamic variables \cite{FV}. Herein, we develop a  new numerical scheme that replicates this positivity property of the regularized Navier-Stokes equations \eqref{eq:regNS_Curvilinear} at the discrete level.

\section{Discrete operators}
\label{SBP}
Similar to the first-order scheme developed in one spatial dimension in \cite{UY_1Dlow}, the proposed first-order scheme for the 3-D Navier-Stokes equations is discretized on the same Legendre-Gauss-Lobatto (LGL) points used for high-order spectral collocation operators. 
In the one-dimensional setting, the physical domain is divided into $K$ non-overlapping elements $[x_1^k, x_{N_p}^k
]$ with $K+1$ nonuniformly distributed points, so that $x_1^k=x_{N_p}^{(k-1)}$.
The discrete solution inside each element  is defined on 
the LGL points, ${\bf x}_k = \left[x_1^k, \dots, x_{N_p}^k \right]^\top$, associated with the Lagrange polynomial basis of degree $p=N_p-1$. These local points ${\bf x}_k$ are referred to as solution points. 

The derivatives of the viscous fluxes in \eqref{eq:regNS_Curvilinear} are discretized by high-order spectral collocation operators that satisfy the summation-by-parts (SBP) property \cite{CFNF,SN1}. This mimetic property is achieved 
by approximating the first derivative by using the following discrete operator, $D$:
\begin{equation}
\label{SCO4}
\begin{split}
&
D=\mathcal{P}^{-1} \mathcal{Q},
\\
&
\mathcal{P} = \mathcal{P}^\top,   \quad {\bf v}^\top \mathcal{P} {\bf v} > 0, \quad \forall {\bf v} \ne {\bf 0}, 
\\
&
\mathcal{Q}=B-\mathcal{Q}^\top, \quad B = {\rm diag}(-1,0,\dots,0,1) ,
\end{split}
\end{equation}
where $\mathcal{P}$ is a diagonal mass matrix and $\mathcal{Q}$ is a stiffness matrix.
Only diagonal-norm SBP operators are considered herein, which is critical for proving
the entropy inequality at the discrete level \cite{CFNF}. 

Along with the solution points, we also 
define a set of intermediate points $\bar{\bf x}_k =
\left[\bar{x}_0^k, \dots, \bar{x}_{N_p}^k \right]^\top$ prescribing bounding control volumes around 
each solution point. These points referred to as flux points form  
a complementary grid whose spacing is precisely equal to the diagonal elements of the positive
definite matrix $\mathcal{P}$ in Eq.~\eqref{SCO4}, i.e., $ \Delta \bar{\bf x} = \mathcal{P} {\bf 1}$,
where $ \bar{\bf x} = \left [\bar{x}_0, \dots, \bar{x}_{N_p} \right]^\top$ is a vector of flux points, 
${\bf 1}=[1, \dots, 1]^\top$, 
and $\Delta$ is an $N_p\times (N_p +1)$ matrix corresponding to the two-point backward
difference operator \cite{CFNF,YC3}.
As has been proven in \cite{FCNYS}, 
all discrete SBP derivative operators can be recast into the following telescopic flux form:
\begin{equation}
\label{TF2}
 \mathcal{P}^{-1} \mathcal{Q} {\bf f}  = \mathcal{P}^{-1} \Delta \bar{\bf f},
\end{equation}
where $\bar{\bf f}$ is a high-order flux vector defined at the flux points. Note that the
above telescopic flux form also satisfies the generalized SBP property \cite{FCNYS}.

For unstructured hexahedral grids, the above one-dimensional SBP operators defined on each grid element naturally extend to three spatial dimensions by using tensor product arithmetic.
The multidimensional tensor product operators are defined as follows:
\begin{equation}
\begin{aligned}
D_{\xi^1} &= \left( D_N \otimes I_N \otimes  I_N \otimes I_5 \right)
 , & 
\mathcal{P}_{\xi^1} &= \left( \mathcal{P}_N \otimes I_N \otimes  I_N  \otimes I_5 \right),
\\
\mathcal{P}_{\xi^1,\xi^2} &=   \left( \mathcal{P}_N \otimes \mathcal{P}_N  \otimes  I_N \otimes I_5 \right) 
, & 
\mathcal{P} &=   \left( \mathcal{P}_N \otimes \mathcal{P}_N  \otimes  \mathcal{P}_N \otimes I_5 \right),
\\
\widehat{\mathcal{P}} &=   \left( \mathcal{P}_N \otimes \mathcal{P}_N  \otimes  \mathcal{P}_N  \right) 
, & 
\mathcal{P}_{\perp,\xi^1} &=   \left( I_N \otimes \mathcal{P}_N  \otimes  \mathcal{P}_N \otimes I_5 \right),
\end{aligned}
\end{equation}
with similar definitions for other directions and operators $\mathcal{Q}_{\xi^i}$, $\Delta_{\xi^i}$ and $B_{\xi^i}$.
Also, the following notation is used hereafter: $\mathcal{P}_{ijk} = \mathcal{P}_{i,i} \mathcal{P}_{j,j} \mathcal{P}_{k,k}$ and $\mathcal{P}_{ij} = \mathcal{P}_{i,i} \mathcal{P}_{j,j}$ where $\mathcal{P}_{i,i}$ is the scalar $i$-th diagonal entry of $\mathcal{P}_N$.  

The metric coefficients are also discretized by using the high-order SBP operators (Eq.~(\ref{SCO4}))such that 
the GCL equations given by Eq.~\eqref{eq:GCL_C} are satisfied exactly at the discrete level \cite{TLGCL}.
The discrete approximation of the scalar metric coefficient $J\frac{\partial \xi^l}{\partial x^m}$ at the solution point $\vec{\xi}_{ijk}$ is denoted $\hat{\bf a}^l_m(\vec{\xi}_{ijk})$.  The block diagonal matrix $[\hat{a}^l_m]$ contains blocks with entries  $\hat{\bf a}^l_m(\vec{\xi}_{ijk}) I_{5 \times 5}$  where $I_{5 \times 5}$ is the identity matrix of size 5.  The specific formulas for $\hat{\bf a}^l_m(\vec{\xi}_{ijk})$ are recorded elsewhere (e.g., see \cite{TLGCL, GCLSS2019}).  Note that $\hat{\bf a}^l_m(\vec{\xi}_{ijk})$ is continuous at element interfaces and satisfies the following GCL equations exactly:
\begin{equation}
\label{DISC_GCL}
\sum\limits_{ l=1}^{3} D_{\xi^l}[\hat{a}^l_m] {\bf 1}_5 = {\bf 0}_5,   \quad m=1,2,3.
\end{equation}

\section{Artificial Viscosity}
\label{AV}  
The artificial viscosity coefficient, $\bfnc{\mu}^{AD}$, in Eq.~(\ref{eq:regNS_Curvilinear}) 
is constructed as a function of the entropy equation residual and consists of three major components: 1) entropy residual, 2)  sensor functions, and 3) upper bound of the artificial viscosity coefficient.  
In this section, we make use of the following globally defined parameters:
$\delta = \left(\frac{1}{K} \right)^{\frac{1}{d}}$ and $L^* = \left( \sum\limits_{ i=1}^{K} V_i \right)^{\frac{1}{d}}$,
where $d$ is the spatial dimensionality, $K$ is the total number of elements used, and $V_i$ is the volume on the $i$-th element.  

\subsection{Entropy residual}
\label{sec:entRes}
We directly generalize the 1-D entropy residual presented  in \cite{UY} to three spatial dimensions.
The finite element residual of the entropy equation on the $k$-th element is approximated as follows:
\begin{equation}
\label{EV3}
\begin{array}{ll}
{\bf R} 
& = \sum\limits_{ l=1}^{3} \left(
\mathcal{P}^{-1}_{\xi^l} \Delta_{\xi^l}\hat{\bar{{\bf f}}}_l 
- D_{\xi^l}\hat{{\bf f}}^{(v)}_l - \hat{{\bf g}}_l \right)  \\
& - \sum\limits_{ l=1}^{3} \left(- \widehat{D}_{\xi^l}\hat{{\bf F}}_l 
+ \widehat{D}_{\xi^l}\left({\bf w}^\top \hat{{\bf f}}^{(v)}_l \right)
- ({\bf\Theta}_{\xi^l})^\top \hat{{\bf f}}^{(v)}_l   \right).
\end{array}
\end{equation}
where $\hat{\bar{{\bf f}}}_l$, $\hat{{\bf f}}^{(v)}_l$, $\hat{{\bf F}}_l$, ${\bf\Theta}_{\xi^l}$, and $\hat{\bf g}_l$  are evaluated by using  the 1st-order discrete solution defined on the LGL solution points.
Note that no high-order discrete solution is required to compute the entropy residual given by Eq.~(\ref{EV3}).
We refer the reader to \cite{UY} which discusses the key advantages of approximating the entropy residual by Eq.~(\ref{EV3}). 

\subsection{Sensor functions}
\label{RBsensor}
We detect regions where the discrete solution is under-resolved and minimize the amount of artificial dissipation in regions where the solution is smooth by using the following sensor functions: 1) an entropy residual-based sensor, 2) a compression sensor, and 3) a pressure gradient sensor.
  
To construct the residual-based sensor, we first define an auxiliary pointwise sensor as follows:
\begin{equation}
{\bf r}(\vec{\xi}_{ijl}) = 
\frac{
\frac{{\bf R}(\vec{\xi}_{ijl})}{{\bf J}(\vec{\xi}_{ijl})} }
{ \max \left(\frac{{\bf R}(\vec{\xi}_{ijl})}{{\bf J}(\vec{\xi}_{ijl})},
\bfnc{\eta}(\vec{\xi}_{ijl})\right) },
\end{equation}
where 
\begin{equation}
\begin{array}{ll}
\bfnc{\eta}(\vec{\xi}_{ijl}) 
= 
&
\left[
\kappa \|\nabla_h{\bf \Theta}^5\| {\bf T}
+
\mu \sqrt{{\bf T}} \sqrt{\sum\limits^4_{n=2} \|\nabla_h{\bf \Theta}^n\|}
+ \| {\bf F}\| 
+
\bfnc{\rho} \delta {\bf c}
\right]_{\vec{\xi}_{ijl}} \times
\\
&
\left[
\frac{1}{\mathcal{P}_{i,i}}
+
\frac{1}{\mathcal{P}_{j,j}}
+
\frac{1}{\mathcal{P}_{l,l}}
\right]
\frac{2}{L^*},
\end{array}
\end{equation}
where $\|\cdot\|$ is the magnitude of a vector. 
In the above equation, 
$
\nabla_h{\bf \Theta}^m(\vec{\xi}_{ijk}) = \left[
\begin{array}{c c c}
{\bf \Theta}^m_{x^1}  & {\bf \Theta}^m_{x^2} & {\bf \Theta}^m_{x^3}
\end{array}
\right]^\top_{\vec{\xi}_{ijk}},
$
where
$
{\bf \Theta}_{x^j}^m(\vec{\xi}_{ijk})
$ is the $m$-th component of 
$
{\bf \Theta}_{x^j}
$
at $\vec{\xi}_{ijk}$ and 
$
{\bf \Theta}_{x^j}
$
is the $p^{\rm{th}}$-order approximation of the gradient of the entropy variables \cite{CFNPSY},
and $c$ is the speed of sound.
The entropy residual sensor for the $k$-th element, $Sn^k$, is then defined as follows:
\begin{equation}
\label{ENTROPYRESIDUALSENSOR} 
\begin{split}
Sn^k_0 
= 
\max({\bf r}^k)^{\max(1,\frac{p-1}{p-1.5})},
\quad
Sn^k 
= 
\left\{
\begin{array}{ll}
 Sn^k_0 , & {\rm if} \ Sn^k_0 \geq \max(0.2,\delta) \\
0,                                                                                                                                                                     & {\rm otherwise}.
\end{array}
\right.
\end{split}
\end{equation}

To identify regions where the amount of
artificial dissipation can be reduced without generating spurious oscillations,
we augment the entropy residual sensor with the compression and pressure gradient sensors, 
which are used only if the residual sensor $Sn^k > 0$ on the $k$-th element.

To minimize the amount of artificial dissipation near expansion waves, the compression sensor, $Cn^k$, on the $k$-th element is defined by using the integral of the divergence of the velocity filed over this element
\begin{equation}
\label{COMPRESSIONSENSOR}
\begin{array}{l}
Cn^k
=
(Cn^k_0)^b
\frac{
\arctan{\left[a (Cn^k_0 - Cn_*)\right]} + \frac{\pi}{2}}
{\arctan{\left[ a(1 - Cn_*)\right]} + \frac{\pi}{2}},
\quad
Cn^k_0 
= 
\max 
\left(
\frac{
-{\bf J}\widehat{\mathcal{P}}\left( \nabla_h \cdot \vec{\bfnc{V}} \right)}
{
{\bf J}\widehat{\mathcal{P}} \left| \left( \nabla_h \cdot \vec{\bfnc{V}} \right) \right|
+
\epsilon}
,
0
\right),
\end{array}
\end{equation}
where $\nabla_h \cdot \vec{\bfnc{V}}$ is the approximation of the divergence of the velocity vector, and
$b$, $Cn_*$, and $a$, are tunable parameters that are set to be $0.1$, $0.2$, and $50$ for all problems considered.

The pressure, $Pn^k$, sensor aims to reduce the amount of artificial dissipation near weak shock waves and 
is defined as follows:
\begin{equation}
\label{presSensor}
\begin{array}{ll}
Pn^k
&
=
\max \left( 0, 
\frac{Pn^k_0}{Pn^k_d}
 \right),
\\
Pn^k_0 
&
= 
- 
\sum^{N}_{i,j,k=1}
\mathcal{P}_{ijk}{\bf J}({\vec{\xi}_{ijk}})  \vec{\bfnc{V}}({\vec{\xi}_{ijk}}) 
\cdot
 (\nabla_h {\bf P})({\vec{\xi}_{ijk}}) ,
\\
Pn^k_d
&
=
\epsilon
+
\sum^{N}_{i,j,k=1}
\mathcal{P}_{ijk}{\bf J}({\vec{\xi}_{ijk}}) \| \vec{\bfnc{V}}({\vec{\xi}_{ijk}}) \|
\| (\nabla_h {\bf P})({\vec{\xi}_{ijk}}) \|.
\end{array}
\end{equation}   

\subsection{Local reference grid spacing}
\label{sec:locRefLenVisc}
On each $k$-th element where $Sn^k > 0$, we define a reference grid spacing, $h^k$, that is used to compute an upper bound of the artificial viscosity coefficient, $\mu^k_{\max}$. The technique presented in this section for calculating $h^k$ is used for all test problems considered in Section~\ref{results}.

First, we compute the following array of reference lengths, ${\bf L}^k$, defined at each solution point of the $k$-th element: 
\begin{equation}
\begin{aligned}
{\bf L}^k(\vec{\xi}_{ijl}) 
&=
2\left[  \prod_{m=1}^{3} \left\| D_{m}{\bf x} \right\|^{\bar{\bf E}_m} \right]_{\vec{\xi}_{ijl}},
\quad
\bar{\bf E}_m(\vec{\xi}_{ijl}) 
= \frac{{\bf E}_m(\vec{\xi}_{ijl})}{\sum_{n=1}^3{\bf E}_n(\vec{\xi}_{ijl})},
\\
{\bf E}_m(\vec{\xi}_{ijl}) 
&= 
\left\| 
\left[ \frac{d_1 \vec{\bfnc{V}}}{d_1 \bfnc{x}} \right]
\frac{D_{m}{\bf x}}
{\left\| D_{m}{\bf x} \right\|} 
\right\|_{\vec{\xi}_{ijl}}  + \epsilon, \quad m = 1,2,3,
\end{aligned}
\end{equation} 
where 
$
D_{m}{\bf x}(\vec{\xi}_{ijl})   = \left[
\begin{array}{c c c}
\widehat{D}_{\xi^m}{\bf x}^1
&
\widehat{D}_{\xi^m}{\bf x}^2
& 
\widehat{D}_{\xi^m}{\bf x}^3
\end{array}
\right]^\top_{\vec{\xi}_{ijl}}
$
is a $p^{\rm{th}}$-order approximation of the tangential derivative in the $\xi^m$ direction and
$\left[\frac{d_1 \vec{\bfnc{V}}}{d_1 \bfnc{x}} \right]$ is a block diagonal matrix with entries
\begin{equation}
\left[ \frac{d_1 \vec{\bfnc{V}}}{d_1 \bfnc{x}} \right](\vec{\xi}_{ijl})
 =
 \left[
\begin{array}{c c c}
 \frac{d_1 \bfnc{V}_1}{d_1 x^1} & \frac{d_1 \bfnc{V}_1}{d_1 x^2} & \frac{d_1 \bfnc{V}_1}{d_1 x^3}  \\
 \frac{d_1 \bfnc{V}_2}{d_1 x^1} & \frac{d_1 \bfnc{V}_2}{d_1 x^2} & \frac{d_2 \bfnc{V}_2}{d_1 x^3}  \\
 \frac{d_1 \bfnc{V}_3}{d_1 x^1} & \frac{d_1 \bfnc{V}_3}{d_1 x^2} & \frac{d_3 \bfnc{V}_3}{d_1 x^3}
\end{array}
\right]_{\vec{\xi}_{ijl}} .
\end{equation} 
Then, an average grid spacing on the $k$-th element is evaluated as follows:
\begin{equation}
\hat{h}^k  = 
\frac{{\bf 1}_1^\top \widehat{\mathcal{P}} {\bf L}^k}{{\bf 1}_1^\top \widehat{\mathcal{P}} {\bf 1}_1}.
\end{equation} 
Finally, the local reference grid spacing, $h^k$, is given by
\begin{equation}
\begin{array}{ll}
h^k 
&
=
\left\{
\begin{array}{ll}
\left[
 \prod\limits_{i \in N_k}h^{\rm v}_i 
\right]^{\frac{1}{|N_k|}} 
  , & {\rm if} \ |N_k| > 0, \\
0,                                                                                                                                                                     & {\rm otherwise}
\end{array}
\right. 
\quad
h^{\rm v}_i  
= 
\left\{
\begin{array}{ll}
\left[
 \prod\limits_{j \in I_i}\hat{h}^j
\right]^{\frac{1}{|I_i|}} 
  , & {\rm if} \ |I_i| > 0 \\
0,                                                                                                                                                                     & {\rm otherwise,}
\end{array}
\right. 
\end{array}
\end{equation}
where $I_i$ is a set of indices of all elements that touch the $i$-th global vertex and have a nonzero $\hat{h}^k$, 
and $N_k$ is a set of indices of all vertices that touch the $k$-th element and have a nonzero $h^{\rm v}_i$.

\subsection{Artificial viscosity coefficient}
As in \cite{UY}, the artificial viscosity for each element is constructed by first finding its physics-based upper bound, $\mu_{\max}$. The elementwise scalar function, $\mu_{\max}$, is constructed such that it is proportional to two-point jumps in the maximum eigenvalue of the inviscid flux Jacobian.
If $Sn^k = 0$ on the $k$-th element, we set $\mu^k_{\max} = 0$.
If $Sn^k > 0$, then $\mu^k_{\max}$ is given by
\begin{equation}
\label{muMax}
\begin{array}{ll}
\mu^k_{\max} 
&=
\frac{\left( h^k \right)^2}{p}
\frac{3(\gamma+1)}{32\gamma} 
z_{Sn^k}
\max\limits_{1 \leq i,j,l \leq N} 
\left[
z_{Pn^k,Cn^k}
\sqrt{
\bar{\bfnc{\rho}}
\sum\limits_{ m=1}^{3} \left( \frac{d_1 {\bf \sqrt{\gamma P}}}{d_1 x^m} \right)^2}
\right.
\\
&+
\left.
\bar{\bfnc{\rho}}
\left(
\frac{\mathcal{P}_{1,1}}{2}
\sqrt{\sum\limits_
{ 
\substack{n,m=1,\\n \ne m}}^{3}
\left(\frac{d_1 \bfnc{V}_n}{d_1 x^m} \right)^2 }
+
\min({\bf Ma},z_{Cn^k})
\left|
\sum\limits_
{a = 1}^{3}
\frac{d_1 \bfnc{V}_a}{d_1 x^a} 
\right|
\right)
\right]_{\vec{\xi}_{ijl}} ,
\\
\bar{\bfnc{\rho}}(\vec{\xi}_{ijl}) &= 
\left(
\bfnc{\rho}(\vec{\xi}_{ijl})
\bfnc{\rho}(\vec{\xi}_{i+1jl})
\bfnc{\rho}(\vec{\xi}_{i-1jl})
\bfnc{\rho}(\vec{\xi}_{ij+1l})
\bfnc{\rho}(\vec{\xi}_{ij-1l})
\bfnc{\rho}(\vec{\xi}_{ijl+1})
\bfnc{\rho}(\vec{\xi}_{ijl-1})
\right)^{\frac17} ,
\\
{\bf Ma}(\vec{\xi}_{ijl}) &=
\frac{\|\vec{\bfnc{V}}(\vec{\xi}_{ijl})\|}{{\bf c}(\vec{\xi}_{ijl})} ,
\\
z_{Sn^k} &= \min(0.5,1.25(Sn^k-0.2)) \geq 0, \quad z_{Cn^k} = \frac{\mathcal{P}_{1,1}}{2}(1-Cn^k) + Cn^k,
\\
z_{Pn^k,Cn^k} &= \min(\frac{\mathcal{P}_{1,1}}{2}(1-Pn^k) + Pn^k, z_{Cn^k}),
\end{array}
\end{equation}
where $h^k$ is a reference grid spacing defined in Section~\ref{sec:locRefLenVisc}, $\mathcal{P}_{1,1}$ is the smallest distance between flux points of the 1-D computational element, and
${\bf c}(\vec{\xi}_{ijk})$ is the speed of sound at $\vec{\xi}_{ijk}$.  Note that $\mu^k_{\max}$ is constructed such that it achieves its maximum value at shocks. 

The velocity gradient components in Eq.~(\ref{muMax}) are approximated by using the following two-point discretizations:
\begin{equation}
\frac{d_1 \bfnc{V}}{d_1 \xi^m}(\vec{\xi}_{i}) =\left\{
\begin{array}{lll}
\frac{\bfnc{V}(\vec{\xi}_{i+1})-\bfnc{V}(\vec{\xi}_{i})}{\xi^m_{i+1}-\xi^m_{i}} , & {\rm if} \ i \leq \frac{p+1}{2} 
\\
\frac{\bfnc{V}(\vec{\xi}_{i+1})-\bfnc{V}_a(\vec{\xi}_{i-1})}{\xi^m_{i+1}-\xi^m_{i-1}},                                                                                                                                                                     & {\rm if} \ \frac{p+1}{2} < i \leq \frac{p+1}{2} + 1
\\
\frac{\bfnc{V}(\vec{\xi}_{i})-\bfnc{V}(\vec{\xi}_{i-1})}{\xi^m_{i}-\xi^m_{i-1}} , & {\rm otherwise} 
\end{array}
\right.,
\end{equation}
for $m = 1, 2, 3$. At element interfaces, we also include velocity jumps formed by the collocated states of elements that share the entire face.  
The gradient of the $\sqrt{\gamma \bf P}$ term in Eq.~(\ref{muMax}) is calculated in a similar fashion.

The globally continuous artificial viscosity $\bfnc{\mu}_k^{AD}$ is then constructed by using the following smoothing procedure.
At each element vertex, we from a unique vertex viscosity coefficient, $\mu^{\rm ver}_i  = \max\limits_{k \in I_i}\mu^k_{\max}$,
where $I_i$ contains indices of all elements that share the $i$-th grid vertex.  After that, 
the globally continuous artificial viscosity is obtained by using the tri-linear interpolation of 8 vertex viscosities, $\mu^{\rm ver}_i$, of the given hexahedral element.

\section{First-order positivity-preserving scheme}
\label{POSPRE} 

We now present a positivity-preserving, entropy stable first-order scheme for the regularized $3$-D compressible Navier-Stokes equations~\eqref{eq:regNS_Curvilinear}. 
Hereafter,
${\bm \nu}_i  = \left[ 
\begin{array}{ccc}
\rho_i &  \vec{\bfnc{V}}_i & T_i 
\end{array}
\right]^\top
$ 
is used to denote the vector of primitive variables  at the $i$-th solution point. We also make substantial use of the logarithmic, harmonic, arithmetic, and geometric averages, which are denoted  for quantities $z_1$ and $z_2$ by using the following subscript notation: $z_{L}$, $z_{H}$, $z_{A}$, and $z_{G}$, respectively.  Note that the following inequalities hold for any $z_1 > 0$, $z_2 > 0$:
$z_H/2 < \min(z_1, z_2) \leq z_H \leq z_G \leq z_L \leq z_A \leq \max(z_1, z_2)$.

\subsection{First-order scheme}
The first-order scheme on a given element is approximated on the same LGL points used for the high-order scheme.  The first-order element treats solution points in a finite volume manner with the flux points acting as control volume edges and can be written as
\begin{equation}
\label{semiDiscCurvi1stOrder} 
\hat{{\bf U}}_t
+ 
\sum\limits_{ l=1}^{3} 
\mathcal{P}^{-1}_{\xi^l} \Delta_{\xi^l}
\left[ 
\hat{\bar{{\bf f}}}^{(in)}_l 
-
\hat{\bar{{\bf f}}}^{(AD_1)}_{\hat{\bar{\sigma}},l}
-
\hat{\bar{{\bf f}}}^{(AD_1)}_l 
\right] 
-
D_{\xi^l}\hat{{\bf f}}^{(v)}_l 
= 
\sum\limits_{ l=1}^{3} \mathcal{P}^{-1}_{\xi^l} 
\left[ 
\hat{{\bf g}}_l + \hat{{\bf g}}^{(AD_1)}_l
\right],
\end{equation} 
where $\hat{\bar{{\bf f}}}^{(in)}_l$ and $\hat{\bar{{\bf f}}}^{(AD_1)}_{\hat{\bar{\sigma}},l}$, $\hat{\bar{{\bf f}}}^{(AD_1)}_l$
are first-order inviscid and artificial dissipation fluxes,
$\hat{{\bf f}}^{(v)}_l, l=1,2,3$, are  the high-order physical fluxes, and $\hat{{\bf g}}_l$ are high-order penalties that are identical to those used in \cite{CFNPSY}.  The way how first-order inviscid and artificial dissipation fluxes are constructed is discussed next.

\subsection{First-order inviscid term}
\label{sec:invLow}
The inviscid fluxes in Eq.~(\ref{semiDiscCurvi1stOrder}) are represented as the sum of entropy conservative and entropy dissipative terms: $\hat{\bar{{\bf f}}}^{(in)}_l = \hat{\bar{{\bf f}}}^{(EC)}_l - \hat{\bar{{\bf f}}}^{(ED)}_l$.  
The exact form of $\hat{\bar{{\bf f}}}^{(ED)}_l$ is presented in Section~\ref{MR_ROE}.  

The entropy conservative flux, $\hat{\bar{{\bf f}}}^{(EC)}_l$, is defined as follows:
\begin{equation}
\label{1ST_ECFLUX}
\left\{
\begin{array}{ll}
\hat{\bar{{\bf f}}}^{(EC)}_1(\vec{\xi}_{\overline{i}})
 = \bar{{ f}}_{(S)}({\bf U}(\vec{\xi}_{i}),{\bf U}(\vec{\xi}_{i+1})) \hat{\bar{\vec{{\bf a}}}}^1(\vec{\xi}_{\overline{i}}), 
& \text{for} \, \, 1 \leq i \leq N-1,\\
\hat{\bar{{\bf f}}}^{(EC)}_1(\vec{\xi}_{\overline{i}}) =  \bar{{ f}}_{(S)}({\bf U}(\vec{\xi}_{\overline{i}}),{\bf U}(\vec{\xi}_{\overline{i}})) \hat{\vec{{\bf a}}}^1(\vec{\xi}_{\overline{i}}), & \text{for} \, \, \overline{i} \in \{0,N \}, 
\end{array}
\right.
\end{equation}
where 
$\hat{\bar{\vec{{\bf a}}}}^1(\vec{\xi}_{\overline{i}}) 
=\sum\limits_{ R=i+1}^{N} \sum\limits_{ L=1}^{i} 2 q_{L,R}
\frac{\hat{\vec{{\bf a}}}^1(\vec{\xi}_{L}) +\hat{\vec{{\bf a}}}^1(\vec{\xi}_{R})}{2}$ and
$\bar{{ f}}_{(S)}(\cdot,\cdot)$ is any two-point, entropy conservative inviscid flux.  
For any two admissible states ${\bm u}_1$ and ${\bm u}_2$, this two--point flux $\bar{{f}}_{(S)}(\cdot,\cdot)$  satisfies the following  condition \cite{TAD2003}:
\begin{equation}
\label{2PT_EC}
\left({\bm w}_1 - {\bm w}_2 \right)^\top \bar{{f}}_{(S)}({\bm u}_1,{\bm u}_2) =  \vec{\bm \psi}_1 - \vec{\bm \psi}_2 .
\end{equation}
In the present analysis, we use the entropy conservative flux developed in \cite{Chand}.
Comparing Eq.~\eqref{1ST_ECFLUX} 
with the high-order entropy stable flux in \cite{CFNPSY},
we note that they are equivalent at the element faces ($\overline{i} \in \{0,N \}$) and only differ at the interior points. 
The high-order interpolation of the metric terms to the flux points is sufficient for proving Lemma~\ref{LEMMA_2PTECFLUX_FREE_and_SS}.
\begin{lemma}
\label{LEMMA_2PTECFLUX_FREE_and_SS}
The inviscid flux $\hat{\bar{{\bf f}}}^{(EC)}_l$ given by Eq. \eqref{1ST_ECFLUX} is freestream preserving and entropy conservative, so that the following equation holds:
\begin{equation}
\begin{aligned} 
\label{ECFLUX_SS_CONTRIB_1STORDER}
\sum\limits_{ l=1}^{3}  
{\bf w}^\top
\mathcal{P}
\mathcal{P}^{-1}_{\xi^l}\Delta_{\xi^l}\hat{\bar{{\bf f}}}^{(EC)}_l 
= \sum\limits_{ l=1}^{3}
{\bf 1}_1^\top
\widehat{\mathcal{P}}_{\perp,\xi^l} 
\widehat{B}_{\xi^l}\hat{{\bf F}}_l. 
\end{aligned}
\end{equation}
Hence, $\hat{\bar{{\bf f}}}^{(EC)}_l$ given by Eq.~\eqref{1ST_ECFLUX} has the same total entropy contribution on each element as the high-order entropy consistent flux in \cite{CFNPSY}.
\end{lemma}
\begin{proof}
To prove the freestream preservation, we show that the following equation holds for any admissible constant state, ${\bm u}_{0}$:
$\sum\limits_{ l=1}^{3} 
\mathcal{P}^{-1}_{\xi^l} 
\Delta_{\xi^l}
\hat{\bar{{\bf f}}}^{(EC)}_l 
=
{\bf 0}_5.
$
Note that  
$
\bar{{ f}}_{(S)}({\bf U}(\vec{\xi}_{ijm}),{\bf U}(\vec{\xi}_{kln}))
=
\bar{{ f}}_{(S)}({\bm u}_{0},{\bm u}_{0})
=
f({\bm u}_{0})
$,
for any two solution points $\vec{\xi}_{ijm}$ and $\vec{\xi}_{kln}$ on this element.
Let us evaluate  
$
\sum\limits_{ l=1}^{3} 
\mathcal{P}^{-1}_{\xi^l} 
\Delta_{\xi^l}
\hat{\bar{{\bf f}}}^{(EC)}_l 
$
at point $\vec{\xi}_{ijm}$ on the element:
\begin{equation}
\begin{array}{ll}
&\left[
\sum\limits_{ l=1}^{3} 
\mathcal{P}^{-1}_{\xi^l} 
\Delta_{\xi^l}
\hat{\bar{{\bf f}}}^{(EC)}_l 
\right]({\vec{\xi}_{ijk}})
\\
&=
f({\bm u}_{0})
\left[
\frac{
\hat{\bar{\vec{{\bf a}}}}^1(\vec{\xi}_{\overline{i}jk}) 
-
\hat{\bar{\vec{{\bf a}}}}^1(\vec{\xi}_{\overline{i-1}jk}) 
}
{\mathcal{P}_{i,i}}
+
\frac{
\hat{\bar{\vec{{\bf a}}}}^2(\vec{\xi}_{i\overline{j}k}) 
-
\hat{\bar{\vec{{\bf a}}}}^2(\vec{\xi}_{i\overline{j-1}k})}
{\mathcal{P}_{j,j}}
+
\frac{
\hat{\bar{\vec{{\bf a}}}}^3(\vec{\xi}_{ij\overline{k}}) 
-
\hat{\bar{\vec{{\bf a}}}}^3(\vec{\xi}_{ij\overline{k-1}})}
{\mathcal{P}_{k,k}}
\right]
\\
&=
f({\bm u}_{0})
\left[
\frac{
\sum\limits_{n=1}^{N}
q_{i,n}
\hat{\vec{{\bf a}}}^1(\vec{\xi}_{njk}) 
}{\mathcal{P}_{i,i}}
+
\frac{
\sum\limits_{n=1}^{N}
q_{j,n}
\hat{\vec{{\bf a}}}^2(\vec{\xi}_{ink}) 
}{\mathcal{P}_{j,j}}
+
\frac{
\sum\limits_{n=1}^{N}
q_{k,n}
\hat{\vec{{\bf a}}}^3(\vec{\xi}_{ijn}) 
}{\mathcal{P}_{k,k}}
\right]
= \left[0, \ldots, 0\right]^\top,
\end{array}
\end{equation}
where the last equality follows from the fact that the metric coefficients satisfy the discrete GCL equation given by Eq.~\eqref{DISC_GCL}.

Let us now show that Eq.~(\ref{ECFLUX_SS_CONTRIB_1STORDER}) holds.
\begin{equation}
\begin{array}{ll}
\sum\limits_{ l=1}^{3} 
{\bf w}^\top \mathcal{P}
\mathcal{P}^{-1}_{\xi^l} 
\Delta_{\xi^l}
\hat{\bar{{\bf f}}}^{(EC)}_l
&=
\sum\limits_{ l=1}^{3} 
{\bf w}^\top \mathcal{P}_{\perp,\xi^l} 
\Delta_{\xi^l}
\hat{\bar{{\bf f}}}^{(EC)}_l
\\
&=
\sum\limits_{ l=1}^{3} 
\widehat{\bf 1}_1^\top
\widehat{\mathcal{P}}_{\perp,\xi^l} 
\left[
{\bf w}^\top 
\Delta_{\xi^l}
\hat{\bar{{\bf f}}}^{(EC)}_l
\right]
\\
&=
\sum\limits_{ l=1}^{3} 
\widehat{\bf 1}_1^\top
\widehat{\mathcal{P}}_{\perp,\xi^l}
\left[
{\bf w}^\top 
\tilde{B}_{\xi^l}
\hat{\bar{{\bf f}}}^{(EC)}_l
-
{\bf w}^\top 
\tilde{\Delta}_{\xi^l}
\hat{\bar{{\bf f}}}^{(EC)}_l
\right]
\\
&=
\sum\limits_{ l=1}^{3} 
\widehat{\bf 1}_1^\top
\widehat{\mathcal{P}}_{\perp,\xi^l}
\left[
\widehat{B}_{\xi^l}
\left(
\hat{\bfnc{\psi}}_l 
+
\hat{{\bf F}}_l  
\right)
-
{\bf w}^\top 
\tilde{\Delta}_{\xi^l}
\hat{\bar{{\bf f}}}^{(EC)}_l
\right].
\end{array}
\end{equation}
Hence, it only remains to show that 
$
\sum\limits_{ l=1}^{3} 
\widehat{\bf 1}_1^\top
\widehat{\mathcal{P}}_{\perp,\xi^l}
\left[
\widehat{B}_{\xi^l}
\hat{\bfnc{\psi}}_l 
-
{\bf w}^\top 
\tilde{\Delta}_{\xi^l}
\hat{\bar{{\bf f}}}^{(EC)}_l
\right]
=
0
$.
\begin{equation}
\begin{array}{ll}
&
\sum\limits_{ l=1}^{3} 
\widehat{\bf 1}_1^\top
\widehat{\mathcal{P}}_{\perp,\xi^l}
\left[
\widehat{B}_{\xi^l}
\hat{\bfnc{\psi}}_l 
-
{\bf w}^\top 
\tilde{\Delta}_{\xi^l}
\hat{\bar{{\bf f}}}^{(EC)}_l
\right]
=
\sum\limits_{j,m=1}^{N}
\widehat{\mathcal{P}}_{\perp,\xi^1}({\vec{\xi}_{ijm}})
\\
&
\left[
\hat{\bfnc{\psi}}_1({\vec{\xi}_{N}})
-
\hat{\bfnc{\psi}}_1({\vec{\xi}_{1}})
-
\sum\limits_{i=1}^{N-1}
\hat{\bar{{\bf f}}}^{(EC)}_1({\vec{\xi}_{i}})
\left(
{\bf w}({\vec{\xi}_{i+1}})
-
{\bf w}({\vec{\xi}_{i}})
\right)^\top 
\right]_{{\vec{\xi}_{\cdot jm}}} + \cdots
\\
&=
\sum\limits_{j,m=1}^{N}
\widehat{\mathcal{P}}_{\perp,\xi^1}({\vec{\xi}_{ijm}})
\left[
\hat{\bfnc{\psi}}_1({\vec{\xi}_{N}})
-
\hat{\bfnc{\psi}}_1({\vec{\xi}_{1}})
-
\sum\limits_{i=1}^{N-1}
\left(
\vec{\bfnc{\psi}}({\vec{\xi}_{i+1}})
-
\vec{\bfnc{\psi}}({\vec{\xi}_{i}}) 
\right)
\hat{\bar{\vec{{\bf a}}}}^1(\vec{\xi}_{\overline{i}}) 
\right]_{{\vec{\xi}_{\cdot jm}}} + \cdots
\\
&=
\sum\limits_{j,m=1}^{N}
\widehat{\mathcal{P}}_{\perp,\xi^1}({\vec{\xi}_{ijm}})
\left[
\sum\limits_{i=1}^{N}
\vec{\bfnc{\psi}}({\vec{\xi}_{i}}) 
\hat{\bar{\vec{{\bf a}}}}^1(\vec{\xi}_{\overline{i}}) 
-
\sum\limits_{i=1}^{N}
\vec{\bfnc{\psi}}({\vec{\xi}_{i}}) 
\hat{\bar{\vec{{\bf a}}}}^1(\vec{\xi}_{\overline{i-1}})
\right]_{{\vec{\xi}_{\cdot jm}}} + \cdots
\\
&=
\sum\limits_{i,j,m=1}^{N}
\widehat{\mathcal{P}}({\vec{\xi}_{ijm}})
\vec{\bfnc{\psi}}({\vec{\xi}_{ijm}})
\left[
\frac{
\hat{\bar{\vec{{\bf a}}}}^1(\vec{\xi}_{\overline{i}jm}) 
-
\hat{\bar{\vec{{\bf a}}}}^1(\vec{\xi}_{\overline{i-1}jm}) 
}
{\mathcal{P}_{i,i}}
+
\frac{
\hat{\bar{\vec{{\bf a}}}}^2(\vec{\xi}_{i\overline{j}m}) 
-
\hat{\bar{\vec{{\bf a}}}}^2(\vec{\xi}_{i\overline{j-1}m})}
{\mathcal{P}_{j,j}}
+
\ldots
\right] 
=0, 
\end{array}
\end{equation}
where 
$\vec{\bfnc{\psi}}({\vec{\xi}_{i}}) 
= 
\left[ 
\bfnc{\psi}_{x^1}({\vec{\xi}_{i}}), \bfnc{\psi}_{x^2}({\vec{\xi}_{i}}), \bfnc{\psi}_{x^3}({\vec{\xi}_{i}}) 
\right]^\top$. The last equality in the above equation again follows from the discrete GCL, Eq.~\eqref{DISC_GCL}.  
\end{proof}

\subsection{First-order artificial dissipation}
The first-order artificial dissipation fluxes and penalties, $\hat{\bar{{\bf f}}}^{(AD_1)}_{\hat{\bar{\sigma}},l}$, $\hat{\bar{{\bf f}}}^{(AD_1)}_l$ and $\hat{{\bf g}}^{(AD_1)}_l$,  in Eq.~\eqref{semiDiscCurvi1stOrder} are constructed as follows.
We begin by presenting the following lemma.
\begin{lemma}
\label{lem:dvdw2pt}
For any two vectors of conservative variables ${\bm u}_1$ and ${\bm u}_2$ with positive density and temperature values, the following  matrix, 
\begin{equation}
\label{dVdW_2PT}
{ \nu}_{ w}({\bm u}_1,{\bm u}_2)
=
\left[
\begin{array}{ccccc}
\frac{\rho_L}{R} & \frac{\rho_L}{R}(V_1)_A & \frac{\rho_L}{R}(V_2)_A  & \frac{\rho_L}{R}(V_3)_A  & \frac{\rho_L}{R} \E_{avg} \\
    0                   &         T_H                      &               0                     &                  0                 &        T_H (V_1)_A              \\
    0                   &         0                          &               T_H                 &                  0                 &        T_H (V_2)_A              \\
    0                   &         0                          &               0                     &                T_H               &        T_H (V_3)_A              \\
    0                   &         0                          &               0                     &                  0                 &        T_G^2                      \\
\end{array}
\right],
\end{equation}
is 1) consistent with 
$\frac{\partial \bfnc{\nu}}{\partial {\bm w}}$, 2) invertible and 3) satisfies the exact algebraic relation 
${ \nu}_{ w}({\bm u}_1,{\bm u}_2)
\left( 
{\bm w}_2 - {\bm w}_1
\right)
=
\left( 
{\bm \nu}_2 - {\bm \nu}_1
\right)
$,
where 
$
\E_{avg} = 
\frac{T^2_G}{T_L}\frac{R}{\gamma-1} 
+ 
\frac{
\vec{\bfnc{V}}({\bm u}_1)
\cdot 
\vec{\bfnc{V}}({\bm u}_2)
}{2}
$,
and $\vec{\bfnc{V}}({\bm u})$ is the velocity associated with state ${\bm u}$.
\end{lemma}
\begin{proof}
The consistency immediately follows from the direct comparison of Eq.~(\ref{dVdW_2PT}) with its continuous counterpart and the consistency of all averages used in ${ \nu}_{ w}({\bm u}_1,{\bm u}_2)$. The second statement is a direct consequence of the fact that ${ \nu}_{ w}({\bm u}_1,{\bm u}_2)$ is a upper triangular matrix with positive diagonal entries. The last statement can be verified directly by calculating the matrix vector product and comparing it with the vector on the right-hand side.
\end{proof}
\begin{remark}
 We denote the inverse of ${ \nu}_{ w}({\bm u}_1,{\bm u}_2)$ as 
${ w}_{ \nu}({\bm u}_1,{\bm u}_2)$, for which the following equality holds 
${ w}_{ \nu}({\bm u}_1,{\bm u}_2)
\left( 
{\bm \nu}_2 - {\bm \nu}_1
\right)
=
\left( 
{\bm w}_2 - {\bm w}_1
\right)
$.
\end{remark}

Using the matrix ${ \nu}_{ w}({\bm u}_1,{\bm u}_2)$, we now prove the following lemma. 
\begin{lemma}
\label{BREN_2PT}
Let $\vec{\bm n}$ be a non-zero direction vector, 
$\vec{\bar{\bm n}}
=
\frac{\vec{\bm n}}{\|\vec{\bm n}\|}
= 
\left[
\bar{n}_1, \bar{n}_2, \bar{n}_3
\right]^\top
$.
For any two admissible states ${\bm u}_1$ and ${\bm u}_2$ with positive density and temperature values,
\begin{equation}
\label{2PTBREN_PRIMMAT}
\begin{split}
&
c_{\nu}^{(B)}({\bm u}_1,{\bm u}_2,\vec{\bm n})
=
\|\vec{\bm n}\|^2
\left[
\begin{array}{ccc}
\sigma & \vec{\bfnc{0}}^\top   &  0 
\\
\sigma \vec{\bfnc{V}}_A & \mu\bar{\mathscr{N}}   &  \vec{\bfnc{0}}
\\
\sigma \E_{avg} 
& 
\mu\left(
\vec{\bfnc{V}}_A^\top + \frac{\vec{\bar{\bm n}}^\top}{3} \vec{\bfnc{V}}_A \cdot \vec{\bar{\bm n}} 
\right) 
& 
\kappa
\end{array}
\right],
\end{split}
\end{equation}
satisfies the following properties:
\begin{equation}
\begin{array}{ll}
&c^{(B)}({\bm u}_1,{\bm u}_2,\vec{\bm n})
=
c_{\nu}^{(B)}({\bm u}_1,{\bm u}_2,\vec{\bm n})
\nu_w({\bm u}_1,{\bm u}_2),
\\
&c^{(B)}({\bm u}_1,{\bm u}_1,\vec{\bm n})
=
\sum\limits_{l,m=1}^3
n_l\Cij{l}{m}^{(B)}({\bm u}_1)n_m,
\quad
c^{(B)}({\bm u}_1,{\bm u}_2,\vec{\bm n}) 
= 
(c^{(B)}({\bm u}_1,{\bm u}_2,\vec{\bm n}))^\top,
\\
&\bm{V}^\top
c^{(B)}({\bm u}_1,{\bm u}_2,\vec{\bm n}) 
\bm{V}  
> 0, 
\quad
\forall  \bm{V}\in\mathbb{R}^{5},   \bm{V} \ne \{{\bm 0}\},
\end{array}
\end{equation} 
where $\bar{\mathscr{N}}_{i,j} = \delta_{i,j} + \frac{\bar{n}_i\bar{n}_j}{3}$, $\E_{avg}$ is defined in Lemma~\ref{lem:dvdw2pt}, $\vec{\bfnc{V}}_A$ is the arithmetic average of the velocities, and $\sigma$, $\mu$, and $\kappa$ are the positive diffusion coefficients of the artificial dissipation flux, and
$\Cij{l}{m}^{(B)}$ are the viscosity matrices associated with the corresponding Cartesian artificial dissipation fluxes. 
\end{lemma}
\begin{proof}
 To show positive definiteness of $c^{(B)}({\bm u}_1,{\bm u}_2,\vec{\bm n})$ for positive diffusion coefficients, we use the Cholesky decomposition, thus leading to
$c^{(B)}({\bm u}_1,{\bm u}_2,\vec{\bm n})= LDL^\top$, where $L$ is invertible and
\begin{equation}
\begin{split}
\label{LDL_2PTBREN}
&D = {\rm diag}
\left(
\| \vec{\bm n} \|^2
\left[
\begin{array}{c c c c c}
\rho_L \frac{\sigma}{R}
&
T_H \mu \frac{d_2}{3}
&
T_H \mu \frac{d_3}{3+\bar{n}_1^2}
&
T_H \mu \frac{4}{4 - \bar{n}_3^2}
&
T_G^2 \kappa
\end{array}
\right]
\right),
\end{split}
\end{equation} 
$
d_2 =  \bar{n}^2_1+ 3
$
and 
$
d_3 = 4 \bar{n}_1^4 + 4 \bar{n}_2^4 + 7 \bar{n}_2^2 \bar{n}_3^2 + 3 \bar{n}_3^4 + \bar{n}_1^2 (8 \bar{n}_2^2 + 7 \bar{n}_3^2)
$.
Since the densities, temperatures, and diffusion coefficients are all positive, $D$ has only positive entries and $c^{(B)}({\bm u}_1,{\bm u}_2,\vec{\bm n})$ is positive definite.
\end{proof} 

For all fixed $1 \leq j,l \leq N$ and $\vec{\xi}_{i} = \vec{\xi}_{ijl}$, $1 \leq i \leq N-1$, the $\hat{\bar{{\bf f}}}^{(AD_1)}_{\hat{\bar{\sigma}},1}$, $\hat{\bar{{\bf f}}}^{(AD_1)}_1$ and $\hat{{\bf g}}^{(AD_1)}_1$ terms are defined as follows:
\begin{equation} 
\label{1ST_ADFLUX}
\begin{array}{ll}
&d\bfnc{\nu}_{m,n}
= \frac{
\bfnc{\nu}(\vec{\xi}_{m})
-
\bfnc{\nu}(\vec{\xi}_{n}) 
}
{\sqrt{{\bm J}(\vec{\xi}_{m}){\bm J}(\vec{\xi}_{n}})},
\\
&\hat{\bar{{\bf f}}}^{(AD_1)}_1(\vec{\xi}_{\overline{i}})
 = 
c_{\nu}^{(B)}({\bf U}(\vec{\xi}_{i}),{\bf U}(\vec{\xi}_{i+1}), \hat{\bar{\vec{{\bf a}}}}^1(\vec{\xi}_{\overline{i}}) )
\left. d\bfnc{\nu}_{i+1,i} \middle/ \left(\xi_{i+1}-\xi_{i}\right) \right.,
\\
&\hat{\bar{{\bf f}}}^{(AD_1)}_{\hat{\bar{\sigma}},1}(\vec{\xi}_{\overline{i}})
 = 
\left. c_{\nu}^{(B)}({\bf U}(\vec{\xi}_{i}),{\bf U}(\vec{\xi}_{i+1}), \hat{\bar{\vec{{\bf a}}}}^1(\vec{\xi}_{\overline{i}}),
\hat{\bar{\bfnc{\sigma}}}_1(\vec{\xi}_{\overline{i}}))\right|_{\mu = \kappa = 0}
\left. d\bfnc{\nu}_{i+1,i} \middle/ \left(\xi_{i+1}-\xi_{i}\right) \right.,
\\
&
\hat{\bar{{\bf f}}}^{(AD_1)}_1(\vec{\xi}_{\overline{0}})
=
\hat{\bar{{\bf f}}}^{(AD_1)}_1(\vec{\xi}_{\overline{N}})
=
\hat{\bar{{\bf f}}}^{(AD_1)}_{\hat{\bar{\sigma}},1}(\vec{\xi}_{\overline{0}})
=
\hat{\bar{{\bf f}}}^{(AD_1)}_{\hat{\bar{\sigma}},1}(\vec{\xi}_{\overline{N}})
=
{\bm 0},
\\
&
\hat{{\bf g}}^{(AD_1)}_1(\vec{\xi}_{i})
=
\left(
\hat{{\bf g}}^{(AD_1)}_1(\vec{\xi}_{1})\delta_{1i}
+
\hat{{\bf g}}^{(AD_1)}_1(\vec{\xi}_{N})\delta_{Ni}
\right),
\\
&
\hat{{\bf g}}^{(AD_1)}_1(\vec{\xi}_{1})
=
c_{\nu}^{(B)}({\bf U}(\vec{\xi}_{0}),{\bf U}(\vec{\xi}_{1}), \hat{\bar{\vec{{\bf a}}}}^1(\vec{\xi}_{\overline{0}}) )
\left. d\bfnc{\nu}_{0,1} \middle/ \mathcal{P}_{1,1}  \right.,
\end{array}
\end{equation}
with identical definitions in the other computational directions.  As discussed in Section~\ref{BNS},
the $\mu$, $\sigma$, and $\kappa$ coefficients in $\hat{\bar{{\bf f}}}^{(AD_1)}_l$ are directly proportional to the artificial viscosity,
$\bfnc{\mu}^{AD}$. At the flux points, the artificial viscosity coefficient is evaluated as the arithmetic average of the corresponding 
 $\bfnc{\mu}^{AD}$ values at the neighboring solution points.
The $\hat{\bar{{\bf f}}}^{(AD_1)}_{\hat{\bar{\sigma}},l}$ flux, which is proportional to $\hat{\bar{\bfnc{\sigma}}}_l$, is introduced to add the mass diffusion  to guarantee the positivity of density (see Section~\ref{subSec:PosDens} for further details).
\begin{lemma}
The $\hat{{\bf g}}^{(AD_1)}_l$, $\hat{\bar{{\bf f}}}^{(AD_1)}_{\hat{\bar{\sigma}},l}$ and 
$\hat{\bar{{\bf f}}}^{(AD_1)}_l$ terms 
given by Eq.~\eqref{1ST_ADFLUX} are entropy stable. 
\end{lemma}
\begin{proof}
Contracting the corresponding terms in Eq.~(\ref{semiDiscCurvi1stOrder}) with the entropy variables, the result follows directly from Lemma~\ref{BREN_2PT},
taking into account that the artificial dissipation matrices are SPD.
\end{proof}

\subsection{First-order Merriam--Roe flux}
\label{MR_ROE}
Recall that the inviscid term $\hat{\bar{{\bf f}}}^{(in)}_l$ of Eq.~\eqref{semiDiscCurvi1stOrder} is written as $\hat{\bar{{\bf f}}}^{(in)}_l = \hat{\bar{{\bf f}}}^{(EC)}_l - \hat{\bar{{\bf f}}}^{(ED)}_l$.  
In this section, we construct the entropy dissipative flux $\hat{\bar{{\bf f}}}^{(ED)}_l$.
Often, two-point entropy conservative fluxes are stabilized through the use of Rusanov-type fluxes (e.g., see   \cite{Zhang,ZhangShu}). Note, however, that the Rusanov flux  dissipates each characteristic wave equally regardless of the magnitude of the corresponding eigenvalue associated with this wave, thus making it too dissipative. A less dissipative and more refined approach is to use an entropy dissipative characteristic flux proposed by Merriam in \cite{MER}, which is herein referred to as the Merriam--Roe (MR) flux and given by 
\begin{align} 
\label{2PTROEGENERAL}
{\bm f}^{(MR)}({\bm u}_1,{\bm u}_2,\vec{\bm n}) 
&=
\bar{{ f}}_{(S)}({\bm u}_1,{\bm u}_2)\vec{\bm n} 
-  
M^{\mathcal{Y}}({\bm u}_1,{\bm u}_2,\vec{\bm n}) \Delta {\bm w},
\end{align}  
where $\bar{{ f}}_{(S)}(\cdot,\cdot)$ is any two-point, consistent, entropy conservative inviscid flux, $\Delta{\bm w} = {\bm w}_2 - {\bm w}_1$, and $M^{\mathcal{Y}}({\bm u}_1,{\bm u}_2,\vec{\bm n})$ is a two-point consistent average of the matrix $\frac{1}{2}\mathcal{Y}|\lambda|\mathcal{Y}^{T} $. The matrix  $\mathcal{Y}$ is a matrix composed out of normalized eigenvectors of the flux Jacobian 
${\bm f}_{{\bf W}}({\bf W},\vec{\bm n}) = {\bm f}_{\bf U}({\bf U},\vec{\bm n}) \frac{\partial {\bf U}}{\partial{\bf W}}$, which can be decomposed as follows:
\begin{equation} 
\label{2PTROEGENERAL2}
\begin{split}
&
f'({\bf W},\vec{\bm n}) 
= 
\mathcal{Y} 
\lambda
\mathcal{Y}^{T},
\, \,
\frac{\partial \bm{U}}{\partial \bm{W}} 
= 
\mathcal{Y}  \mathcal{Y}^{T},
\end{split}
\end{equation}  
where 
$
\lambda 
= 
{\rm diag}
\left(
\begin{array}{c c c c c}
-c \| \vec{\bm n} \|  + \vec{\bfnc{V}} \cdot \vec{\bm n}
&
c \| \vec{\bm n} \|  + \vec{\bfnc{V}} \cdot \vec{\bm n}
&
\vec{\bfnc{V}} \cdot \vec{\bm n}
&
\vec{\bfnc{V}} \cdot \vec{\bm n}
&
\vec{\bfnc{V}} \cdot \vec{\bm n}
\end{array}
\right)
$
and $\mathcal{Y}$ can be found in \cite{Fthesis}. 
 The matrix $M^{\mathcal{Y}}({\bm u}_1,{\bm u}_2,\vec{\bm n})$ is SPSD if the density and temperature values used to build the matrix are positive.  For two admissible states ${\bm u}_1$ and ${\bm u}_2$, there are many options for building $M^{\mathcal{Y}}({\bm u}_1,{\bm u}_2,\vec{\bm n})$ at an interface.  In the present analysis, we use the following average:
\begin{equation}
\label{AVGVROE}
{\bm \nu}({\bm u}_1,{\bm u}_2) = 
\left[ 
\begin{array}{ccc}
\rho_L                              &
\frac{\vec{\bfnc{V}}({\bm u}_1)T_2 + \vec{\bfnc{V}}({\bm u}_2) T_1}{T1 + T2}     &
T_H                                 
\end{array}
\right]^\top,
\end{equation}
where $\vec{\bfnc{V}}({\bm u}_1)$ is the velocity vector of ${\bm u}_1$.
With this average, we can write the first component of $M^{\mathcal{Y}}({\bm u}_1,{\bm u}_2,\vec{\bm n}) \Delta {\bm w}$ in a form that facilitates proving the positivity of density:
\begin{equation}
\label{ROEDENSCONTRIB}
\begin{array}{ll}
&
\left(
M^{\mathcal{Y}}({\bm u}_1,{\bm u}_2,\vec{\bm n}) 
\Delta {\bm w}\right)_{\rho} 
= 
\rho_L \mathcal{V}({\bm u}_1, {\bm u}_2,\vec{\bm n})  
+ 
\Delta\rho\lambda_c ,
\\
&
\mathcal{V}({\bm u}_1, {\bm u}_2,\vec{\bm n}) 
=  
-
\left(
\frac{\Delta(\log T)}{\gamma-1}  
+
\frac{
\Delta T 
\left \| \Delta\vec{\bfnc{V}} \right \|^2
}
{8 R_g T_A^2}
\right)\lambda_c 
\\
&
+ 
\frac{\Delta T}{4T_A(\gamma-1)}(\lambda_2 + \lambda_3)  
+ 
(\lambda_3-\lambda_2)
\frac{\Delta \vec{\bfnc{V}} 
\cdot 
\frac{\vec{\bm n}}{\left \| \vec{\bm n} \right \|}
\sqrt{T_H}
}
{2T_A\sqrt{R_g \gamma}},
\end{array}
\end{equation}
where 
$\lambda_1 = |\vec{\bfnc{V}}_{avg}\cdot \vec{\bm n}|$, 
$\lambda_2 
= 
|\vec{\bfnc{V}}_{avg}\cdot \vec{\bm n}
-
c_{avg} \left \| \vec{\bm n} \right \| |$, 
$\lambda_3 
= 
|\vec{\bfnc{V}}_{avg}\cdot \vec{\bm n}
+
c_{avg} \left \| \vec{\bm n} \right \| 
|
$,
$\lambda_c = \frac{\lambda_1(\gamma-1)}{2\gamma} + \frac{\lambda_2+\lambda_3}{4\gamma}$, 
$\vec{\bfnc{V}}_{avg} =  
\frac{\vec{\bfnc{V}}_1T_2 + \vec{\bfnc{V}}_2 T_1}{T1 + T2}   $
, $c_{avg} = \sqrt{R_gT_H\gamma}$, and
$\Delta T = T_2 - T_1$.   

For all fixed $1 \leq j,l \leq N$ and $\vec{\xi}_{i} = \vec{\xi}_{ijl}$,
the $ \hat{\bar{{\bf f}}}^{(ED)}_l$ term in the $\hat{\bar{{\bf f}}}^{(in)}_l$ flux in Eq. \eqref{semiDiscCurvi1stOrder} is defined as follows:  
\begin{equation}
\label{MERRIAM_CONTRIB_INTERIOR}
\hat{\bar{{\bf f}}}^{(ED)}_1(\vec{\xi}_{\overline{i}})
 = 
M^{\mathcal{Y}}({\bf U}(\vec{\xi}_{i}),{\bf U}(\vec{\xi}_{i+1}), \hat{\bar{\vec{{\bf a}}}}^1(\vec{\xi}_{\overline{i}}) )
\left (
{\bf w}(\vec{\xi}_{i+1})
-
{\bf w}(\vec{\xi}_{i})
\right ), \ 1 \leq i \leq N-1
\end{equation}
and 
$
\hat{\bar{{\bf f}}}^{(ED)}_1(\vec{\xi}_{\overline{0}})
=
\hat{\bar{{\bf f}}}^{(ED)}_1(\vec{\xi}_{\overline{N}})
=
{\bm 0}
$
with the identical definitions in the other computational directions.

\subsection{Positivity of density}
\label{subSec:PosDens}
We now discuss how to guarantee the positivity of density for the first-order scheme given by Eq. \eqref{semiDiscCurvi1stOrder}. 
\begin{theorem}
\label{POSDENSTHM}
Assume that the explicit Euler discretization in time is used for the scheme given by Eq.~\eqref{semiDiscCurvi1stOrder} and  $\hat{\bar{{\bf f}}}^{(in)}_l, l=1,2,3$ are some consistent inviscid interface fluxes whose  first components 
can be written as: 
$\hat{\bar{\bm f}}^{\rho\pm}_l = \hat{\bar{m}}^{\pm}_l	- 
 \mathscr{D}^{\pm}_l \Delta^{\pm}_l\rho$ where 
$\Delta^{+}_1\rho= \rho_{i+1jk}-\rho_{ijk}$,
$\Delta^{-}_1\rho= \rho_{ijk}-\rho_{i-1jk}$,
etc.  Let $\rho^{+}_{1,A} = \frac{\rho_{ijk} + \rho_{i+1jk}}{2}$ and $\rho^{-}_{1,A} = \frac{\rho_{ijk} + \rho_{i-1jk}}{2}$ with the identical definitions in other directions. If 
$\mathscr{D}^{\pm}_l 
\geq 
\mathscr{D}^{\pm}_{l,\min}
=
\frac{|\hat{\bar{m}}^{\pm}_l|}{2\rho^{\pm}_{l,A}} $,
 then this first--order FV scheme ( Eq.~\eqref{semiDiscCurvi1stOrder}) preserves the positivity of the density $\rho$ under the following CFL condition: 
\begin{equation}
\label{CFLDENSPOS_SOLNPT}
\tau <
\frac{J_{ijk}}
{2 \sum\limits_{l=1}^{3}
\frac{\mathscr{D}^{+}_l + \mathscr{D}^{-}_l}{P_{ll}}}
=
\tau_{\rho}.
\end{equation}     
\end{theorem}
\begin{proof}
Let $J_{ijk} = {\bf J}(\vec{\xi}_{ijk})$, $\rho_{ijk} = \bfnc{\rho}(\vec{\xi}_{ijk})$, and ${\bf U}_{ijk} = {\bf U}(\vec{\xi}_{ijk})$.  
Consider  the update of the density $J_{ijk}\rho_{ijk} = \hat{\rho}$ at the solution point $\vec{\xi}_{ijk}$, which depends only on the nearest neighbors.  With the following notation for the interface fluxes
$\hat{\bar{\bm f}}^{\rho}_1({\bf U}_{ijk},{\bf U}_{i+1jk})
=
\hat{\bar{\bm f}}^{\rho+}_1
$
,
$\hat{\bar{\bm f}}^{\rho}_1({\bf U}_{ijk},{\bf U}_{i-1jk})
=
\hat{\bar{\bm f}}^{\rho-}_1
$
and similarly in other directions, we have 
\begin{equation}
\label{FVDENS}
\begin{split}
&\hat{\rho}^{n+1} = \hat{\rho}^n 
- 
\tau
\left[
\frac{
\hat{\bar{\bm f}}^{\rho+}_1
- 
\hat{\bar{\bm f}}^{\rho-}_1
}
{P_{ii}}
+
\frac{
\hat{\bar{\bm f}}^{\rho+}_2
- 
\hat{\bar{\bm f}}^{\rho-}_2
}
{P_{jj}}
+
\frac{
\hat{\bar{\bm f}}^{\rho+}_3
- 
\hat{\bar{\bm f}}^{\rho-}_3
}
{P_{kk}}
\right].
\end{split}
\end{equation}
The above density equation can be split as follows:
\begin{equation}
\begin{split}
&\hat{\rho}^{n+1} =
\left(
\frac{\hat{\rho}^n}{6}
- 
\tau
\frac{
\hat{\bar{\bm f}}^{\rho+}_1
}
{P_{ii}}
\right)
+
\left(
\frac{\hat{\rho}^n}{6}
+
\tau
\frac{
\hat{\bar{\bm f}}^{\rho-}_1
}
{P_{ii}}
\right)
+
\ldots
\end{split}
\end{equation}
Since all terms in the above equation have the same structure, we only consider the first term:
$$
\left(
\frac{\hat{\rho}^n}{6}
- 
\tau
\frac{
\hat{\bar{\bm f}}^{\rho+}_1
}
{P_{ii}}
\right)
= 
\frac{\hat{\rho}^n}{6} 
- 
\frac{\tau}{P_{ii}}
(\hat{\bar{m}}^{+}_1	
- 
\mathscr{D}^{+}_1
\Delta^{+}_1\rho).
$$  
We now consider two cases: 1)  $\mathscr{D}^{+}_1 = 0$ and 2)  $\mathscr{D}^{+}_1 > 0$.
If $\mathscr{D}^{+}_1 = 0$, then taking into account that 
$\mathscr{D}^{+}_1 \geq \frac{|\hat{\bar{m}}^{+}_1|}{2\rho^+_{1,A}} $,
we have 
$
\left(
\frac{\hat{\rho}^n}{6}
- 
\tau
\frac{
\hat{\bar{\bm f}}^{\rho+}_1
}
{P_{ii}}
\right) = \frac{\hat{\rho}^n}{6}  
=
\rho^n
\left[
\frac{J_{ijk}}{6} 
- 
\frac{2\tau \mathscr{D}^{+}_1}{P_{ii}}
\right]
$.  
Now, assume that $\mathscr{D}^{+}_1 > 0$ which yields 
\begin{equation}
\label{oneFace_POSDENSTHM}
\begin{split}
\frac{\hat{\rho}^n}{6} 
- 
\frac{\tau}{P_{ii}}
(\hat{\bar{m}}^{+}_1	
- 
\mathscr{D}^{+}_1
\Delta^{+}_1\rho)
&=
\frac{\hat{\rho}^n}{6} 
- 
\frac{\tau \mathscr{D}^{+}_1}{P_{ii}}
\left(
\frac{\hat{\bar{m}}^{+}_1}{\mathscr{D}^{+}_1}
- 
\Delta^{+}_1\rho
\right)
\\
&
\geq
\frac{\hat{\rho}^n}{6} 
- 
\frac{\tau \mathscr{D}^{+}_1}{P_{ii}}
\left(
2\rho^+_{1,A}
- 
\Delta^{+}_1\rho
\right)
\\&
=
\rho^n
\left[
\frac{J_{ijk}}{6} 
- 
\frac{2\tau \mathscr{D}^{+}_1}{P_{ii}}
\right]
.
\end{split}
\end{equation}
Summing over all element interfaces, we have
\begin{equation}
\label{sumFaces_POSDENSTHM}
\begin{split}
\hat{\rho}^{n+1}
\geq
\rho^n
\left[
J_{ijk}
- 
2\tau
\sum\limits_{l=1}^{3}
\frac{\mathscr{D}^{+}_l + \mathscr{D}^{-}_l}{P_{ll}}
\right]
>
0.
\end{split}
\end{equation}
\end{proof}

We now give two examples demonstrating how Theorem~\ref{POSDENSTHM} can be used to preserve density positivity of the scheme given by Eq. \eqref{semiDiscCurvi1stOrder} when only the minimum mass diffusion is used i.e. 
$\mathscr{D}^{\pm}_l = \mathscr{D}^{\pm}_{l,\min} = \frac{|\hat{\bar{m}}^{\pm}_l|}{2\rho^{\pm}_{l,A}} $.

\subsubsection{Positivity of density:  Ismail--Roe EC flux}
Assume that  $\hat{\bar{{\bf f}}}^{(in)}_l = \hat{\bar{{\bf f}}}^{(EC)}_l$, where $\bar{{ f}}_{(S)}(\cdot,\cdot)$ in Eq.~\eqref{1ST_ECFLUX} is the entropy conservative flux of Ismail and Roe \cite{IR}. Therefore, 
$\hat{\bar{m}}^{\pm}_l = 
\gamma 
(\rho c)_L 
\left( \frac{\vec{\bfnc{V}}}{c} \right)_A \cdot \vec{\bm n}$, thus leading to
\begin{equation}
\begin{array}{l}
\mathscr{D}^{\pm}_{l,\min}
=
\frac{
|\gamma 
(\rho c)_L 
\left( \frac{\vec{\bfnc{V}}}{c} \right)_A \cdot \vec{\bm n}|}
{2\rho_A}
\leq
\frac{
\gamma 
(\rho c)_A}
{2\rho_A}
\left| \left( \frac{\vec{\bfnc{V}}}{c} \right)_A \cdot \vec{\bm n} \right |
\leq
\gamma 
c_A
\left| \left( \frac{\vec{\bfnc{V}}}{c} \right)_A \cdot \vec{\bm n} \right |,
\end{array}
\end{equation}
where $c$ is the speed of sound. Note that for $\mathscr{D}^{\pm}_{l,\min}$ presented above,  the time step constraint given by Theorem~\ref{POSDENSTHM} is comparable with that of the conventional CFL condition.

\subsubsection{Positivity of density:  Merriam--Roe flux}
We now use Theorem~\ref{POSDENSTHM} to prove density positivity when the Merriam--Roe flux is used in Eq.~\eqref{semiDiscCurvi1stOrder}.
\begin{corollary}
\label{COR:posDensMassDiff_1stOrder}
Assume that $\hat{\bar{{\bf f}}}^{(in)}_l$ in Eq.~\eqref{semiDiscCurvi1stOrder} is the EC flux of Chandrashekar \cite{Chand}. Let $\hat{\bar{\vec{{\bf a}}}}^l_{\pm}$ be the metric term at the ``$\pm$'' interface in the $l$-th direction.  
If the explicit Euler discretization in time is used in Eq.~\eqref{semiDiscCurvi1stOrder}, then this 1st--order FV scheme preserves the positivity of density under the time step constraint given by Eq.~\eqref{CFLDENSPOS_SOLNPT} with
\begin{equation}
\label{posDensMassDiff_1stOrder}
\left[
\lambda_c + 
\frac{\sigma \| \hat{\bar{\vec{{\bf a}}}} \|^2 }
{J_G \Delta \xi}
\right]^{\pm}_l
=
\mathscr{D}^{\pm}_l 
\geq 
\mathscr{D}^{\pm}_{l,\min}
\\
=
\frac{\rho^{\pm}_{l,L}}{2\rho^{\pm}_{l,A}}
\left
|
\vec{\bfnc{V}}_A \cdot \hat{\bar{\vec{{\bf a}}}}^l_{\pm}
-
\mathcal{V}({\bm u}, {\bm u}^{\pm}_l ,\hat{\bar{\vec{{\bf a}}}}^l_{\pm})
\right
|,
\end{equation} 
 and the following constraint on $\sigma^{\pm}_l$:
\begin{equation}
\label{posDensMassDiff_1stOrder_SigmaConstraint}
\begin{split}
\sigma^{\pm}_l 
\geq 
\sigma^{\pm}_{l,\min}
=
\left[
\max
\left(
0
,
\frac{\rho_L}{2\rho_A}
\left
|
\vec{\bfnc{V}}_A \cdot \hat{\bar{\vec{{\bf a}}}}
-
\mathcal{V}({\bm u}, {\bm u}^{\pm}_l ,\hat{\bar{\vec{{\bf a}}}})
\right
|
- \lambda_c
\right)
\frac{J_G \Delta \xi}{\| \hat{\bar{\vec{{\bf a}}}} \|^2}
\right]^{\pm}_l.
\end{split}
\end{equation}
\end{corollary}
\begin{proof}
The proof follows directly from Theorem~\eqref{POSDENSTHM}.
\end{proof}

\subsection{Positivity of internal energy }
\label{TEMPPOSCONDITION}
If the explicit first-order Euler scheme is used to advance the solution in time, i.e. 
\begin{equation}
\label{explicitE}
\hat{{\bf U}}^{n+1} = \hat{{\bf U}}^{n} + \tau\hat{{\bf U}}_t,
\end{equation}
so that $\tau$ is in the interval that preserves the positivity of $\bfnc{\rho}^{n+1}(\vec{\xi}_{ijk})$, then the positivity of the internal energy at the time level $n+1$ at the solution point $\vec{\xi}_{ijk}$ is solely determined by the following quadratic polynomial in $\tau$:
\begin{equation}
\label{TEMPPOLY}
\text{IE}({\bm u}^{n+1})\rho^{n+1} 
=
\left(\frac{\tau}{J}\right)^2
\left(
\frac{dE}{dt}\frac{d\rho}{dt}-\frac{1}{2} 
\left \| \frac{d{\bm m}}{dt} \right \|^2
\right) 
+ 
\frac{\tau}{J}
\left({{\bm u}}^{n}\right)^\top \left[ 
\begin{array}{l}
\phantom{-} \frac{dE}{dt}     \\
-\frac{d{\bm m}}{dt}    \\
\phantom{-} \frac{d\rho}{dt}  \\
\end{array}
\right] + \text{IE}({\bm u}^{n})\rho^{n},
\end{equation}
where $\hat{{\bf U}}^n(\vec{\xi}_{ijk}) = J {\bm u}^n$,
${\bf J}(\vec{\xi}_{ijk}) = J$,
$\bfnc{\rho}^{n+1}(\vec{\xi}_{ijk}) = \rho^{n+1}$,
$
\hat{{\bf U}}_t(\vec{\xi}_{ijk}) = 
\left[
\frac{d\rho}{dt}, \frac{d{\bm m}}{dt}, \frac{dE}{dt}
\right]^\top
$
and $\text{IE}({\bm u}^{n})$ is the internal energy of ${\bm u}^{n}$. Note that Eq.~(\ref{TEMPPOLY}) holds for any spatial discretization.
Using Eq.~(\ref{TEMPPOLY}), we now prove the following lemma.
\begin{lemma}
\label{ZEROROOTBOUND}
Let the discrete solution at the time level $n$ be in the admissible set, so that
$\bfnc{\rho}^n(\vec{\xi}_{ijk}),\text{IE}({\bf U}^{n}(\vec{\xi}_{ijk}))>0$ for all solution points in the domain.  
Then, there exists $ \tau^{\min} \in (0, \tau^\rho]$, where  $\tau^{\rho}$ is given by Eq.~(\ref{CFLDENSPOS_SOLNPT}), such that  for all $\tau$: $0 < \tau < \tau^{\min}$, the 1st--order FV scheme given by Eqs.~(\ref{semiDiscCurvi1stOrder}) and (\ref{explicitE}) preserves the positivity of internal energy, i.e., $\text{IE}({\bf U}^{n+1}(\vec{\xi}_{ijk}))>0$ at every solution point.
\end{lemma}
\begin{proof}
Since for all solution points $\text{IE}({\bf U}^{n}(\vec{\xi}_{ijk}))\bfnc{\rho}^{n}(\vec{\xi}_{ijk}) > 0$, the above quad-ratic trinomial in $\tau$ is either strictly positive, i.e., $\text{IE}({\bf U}^{n+1}(\vec{\xi}_{ijk}))\bfnc{\rho}^{n+1}(\vec{\xi}_{ijk})>0$, $\forall \tau>0$ 
(thus imposing no time step constraint for positivity of temperature), 
or there exists the minimum positive root $\bfnc{\tau}^{\min}(\vec{\xi}_{ijk})$ of the following quadratic equation: $\text{IE}({\bf U}(\vec{\xi}_{ijk})^{n+1})\bfnc{\rho}(\vec{\xi}_{ijk})^{n+1}=0$, for which the positivity of $\text{IE}({\bf U}(\vec{\xi}_{ijk})^{n+1})\bfnc{\rho}(\vec{\xi}_{ijk})^{n+1}$ is guaranteed for all $\tau < \bfnc{\tau}^{\min}(\vec{\xi}_{ijk})\le \tau_{\rho}$.  Hence, for the scheme given by Eqs.~(\ref{semiDiscCurvi1stOrder}) and \eqref{explicitE}, a sufficient condition for positivity of internal energy at the time level $n+1$  is $\tau < \tau^{\min} = \min(\tau^\rho,\min\limits_{ijk}(\bfnc{\tau}^{\min}(\vec{\xi}_{ijk})))$ (note that if $\tau^{\rho}$ is sharp, then $\tau < \tau^{\min}$ is also a necessary condition).
\end{proof}

To bound the internal energy at each solution point, $\text{IE}({\bf U}^{n+1}(\vec{\xi}_{ijk}))$, from below by some nonzero quantity, we can choose $\tau \le \tau^{\min} = \min\limits_{ijk}(\bfnc{\tau}^{\min}(\vec{\xi}_{ijk}))$, where $\bfnc{\tau}^{\min}(\vec{\xi}_{ijk})$ is redefined as follows.  Let $c_{\text{IE}}$ be a user-defined parameter $0 < c_{\text{IE}} < 1$.  Then, $\bfnc{\tau}^{\min}(\vec{\xi}_{ijk})$ is defined such that ${\text{IE}}({\bf U}^{n+1}(\vec{\xi}_{ijk})) \geq  c_{\text{IE}} \text{IE}({\bf U}^{n}(\vec{\xi}_{ijk}))$. Hence, $\bfnc{\tau}^{\min}(\vec{\xi}_{ijk})$ is the minimum positive root of the following quadratic equation:
\begin{equation}
\label{TEMPPOLY2}
0
=
\left(\frac{\tau}{J}\right)^2
\left(
\frac{dE}{dt}\frac{d\rho}{dt}-\frac{1}{2} 
\left \| \frac{d{\bm m}}{dt} \right \|^2
\right) 
+ 
\frac{\tau}{J}
\left(\widetilde{\bm u}^{n}\right)^\top \left[ 
\begin{array}{l}
\phantom{-} \frac{dE}{dt}     \\
-\frac{d{\bm m}}{dt}    \\
\phantom{-} \frac{d\rho}{dt}  \\
\end{array}
\right] + \text{IE}(\widetilde{\bm u}^{n})\rho^{n},
\end{equation}
where $\widetilde{\bm u}_i^{n}$ is ${\bm u}_i^{n}$ with the temperature scaled by $1-c_{\text{IE}}$.  If no positive roots exist for this equation, then for all $\tau > 0$, $\text{IE}({\bm u}_i^{n+1}) >  c_{\text{IE}} \text{IE}({\bm u}_i^{n})$; otherwise, there exists the minimum positive root $\tau^{\text{min}}_i$, such that for all $\tau \leq \tau^{\text{min}}_i$ the following inequality holds: $\text{IE}({\bm u}_i^{n+1}) \geq  c_{\text{IE}} \text{IE}({\bm u}_i^{n})$.

\subsection{Entropy stability}
\label{ENTSTAB_1stSCHEME}
We now show that the first-order scheme (Eq. (\ref{semiDiscCurvi1stOrder})) is entropy stable.
\begin{theorem}
The semi-discrete first-order scheme given by Eq. (\ref{semiDiscCurvi1stOrder}) 
is entropy stable.
\end{theorem}
\begin{proof}
Entropy stability of the scheme given by Eq. (\ref{semiDiscCurvi1stOrder}) can be proven for the
time derivative, inviscid, viscous, and artificial dissipation terms individually.  Contracting Eq.~(\ref{semiDiscCurvi1stOrder})
with entropy variables and taking into account that the mass matrices are diagonal, 
 the time derivative term 
can be manipulated as ${\bf W}^\top\hat{\mathcal{P}}d(J{\bf U})/dt = {\bf 1}^{\top}\hat{\mathcal{P}}d(J \mathcal{S})/dt$  (e.g., see \cite{CFNPSY}).
The entropy stability of the inviscid terms follows directly from Lemma ~\ref{LEMMA_2PTECFLUX_FREE_and_SS}. The entropy stability of the
high-order viscous terms and the corresponding penalties have been proven in \cite{CFNF, CFNPSY}.
The first-order artificial dissipation terms and their penalties 
$\hat{\bar{{\bf f}}}^{(ED)}_l$, $\hat{\bar{{\bf f}}}^{(AD_1)}_l$,$\hat{\bar{{\bf f}}}
^{(AD_1)}_{\hat{\bar{\sigma}},l}$, and $\hat{{\bf g}}^{(AD_1)}_l$, are all formed by using  
SPSD matrices multiplied by 2-point jumps in the entropy variables and therefore are easily shown to be entropy dissipative.
\end{proof}

\section{ Entropy stable velocity and temperature limiters}
\label{VISLIM}
The high-order discretization of the viscous terms may significantly increase the stiffness of the time step constraint required for temperature positivity. To overcome this problem, we construct new conservative, discretely entropy stable limiters that bound the magnitude of the velocity and temperature gradients in troubled elements.   The proposed approach differs from the limiter in \cite{ZhangShu} in two distinct ways:
1) density is not altered at any solution point and 2) the limiters are applied before negative temperatures are encountered.  The benefit of the latter property is that one can then prove discrete entropy stability.

\subsection{Bounds on velocity and temperature}
\label{FLUXFORSMOOTHCRIT}
Taking into account the contribution of velocity and temperature terms to the high-order approximation of the gradient of entropy variables and consequently to the viscous fluxes, we propose to impose the following bounds on $(V_l)_i$ and $T_i$ at each solution point of a troubled element:
\begin{align}
\label{VTBOUNDS}
 | (V_l)_i-\dbar{V_l} | \leq \frac{\dbar{\rho}_H h \dbar{T}_H}{\mu}, \, \, \, 
 \tilde{\lambda}_i \frac{|T_i - \dbar{T}|}{T_i \dbar{T}}  \leq \frac{\dbar{\rho}_H h }{\mu}, 
\end{align}    
where 
\begin{equation}
\label{LMBD}
\tilde{\lambda}_i = \frac{\| \vec{\bfnc{V}}_i+\dover{\vec{\bfnc{V}}} \|}{2} +  \frac{c_i + c(\dbar{T})}{2},
\end{equation}
$\dbar{q}$ is the arithmetic average of a quantity $q$ on a high-order element,
$\dbar{\rho}_H$ is the harmonic average of $\rho_i$ and $\dbar{\rho}$, $\mu$ is the physical viscosity coefficient, $c(\dbar{T})$ is the speed of sound associated with the average temperature and $h$ is a reference length for the element, e.g., $h=\rm{V}^{1/3}$, where $\rm{V}$ is the element volume. Note that other cell averages on a given high-order element can also be used instead of the arithmetic averages in Eqs.~\eqref{VTBOUNDS} and \eqref{LMBD}.

We now construct velocity and temperature limiters such that  they ensure the bounds given by Eq.~(\ref{VTBOUNDS}) without changing the density at any solution point, preserve the conservation of mass, momentum and energy,  and can only decrease the discrete integral of the mathematical entropy on a given element.
The limiting procedure is broken into two steps.  The first step enforces the velocity bound while altering the temperature field in a pointwise discretely entropy stable manner.  The second step enforces the temperature bound by only altering the energy equation in an elementwise entropy stable manner.

\subsection{A limiter to enforce the velocity bound}
First, we modify the velocity at each solution point on a given high-order element, so that it satisfies Eq.~(\ref{VTBOUNDS}).  Let $\vec{\xi}_{ijk} = \vec{\xi}_{a}$ be some solution point on the element.
To enforce this velocity bound, we propose the following limiter:
\begin{equation}
\label{uThetaVel}
\hat{\bf U}_a^{v} = \hat{\bf U}_a
+ 
\frac{1}{\mathcal{P}_a}{\bm f}_v({\bf U}_a,{\bm \theta}^v),
\end{equation}
where 
$\hat{\bf U}(\vec{\xi}_{ijk})^{v} = \hat{\bf U}_a^{v}$,
$\mathcal{P}_{ijk} = \mathcal{P}_a$,
$
{\bm \theta}^v
=
\left[
\begin{array}{ccc}
\theta^v_1 & \theta^v_2 & \theta^v_3
\end{array}
\right]^\top
$,
\begin{equation}
\label{VelFluxLimOpt2}
{\bm f}_v({\bf U}_a,{\bm \theta}^v) =
\rho_{\min}
\left[ 
\begin{array}{ccc}
0                    & 0                    & 0          \\
\theta^v_1 		     & 0                    & 0          \\
0                    & \theta^v_2	        & 0          \\
0                    & 0          			& \theta^v_3 \\
\theta^v_1\dbar{V_1} & \theta^v_2\dbar{V_2} & \theta^v_3\dbar{V_3}
\end{array}
\right]
\left( \dover{\vec{\bfnc{V}}}-\vec{\bfnc{V}}_a \right),
\end{equation}
$\rho_{\min}$ is the minimum density on the element and  $\dover{\vec{\bfnc{V}}}$ is the arithmetic average of velocity on the high-order element. 

Note that the temperature after applying the velocity limiter is given by
\begin{equation}
\label{tempChangeFnThetaVELSQUEEZE}
T({\bf U}_a^v) 
= 
T({\bf U}_a) 
+
\Delta \vec{\bfnc{V}} M({\bf U}_a,{\bm \theta}^v) \Delta \vec{\bfnc{V}},
\end{equation}
where
\begin{equation}
\label{tempChangeFnThetaVELSQUEEZE_2}
\begin{array}{ll}
M({\bf U}_a,{\bm \theta}^v) 
&=
\frac{\gamma-1}{R_g}\frac{\rho_{\min}}{\rho_a J_a \mathcal{P}_a}
\\
&
{\rm diag}
\left[
\theta^v_1 
\left( 
1 - \frac{\theta^v_1 \rho_{\min}}{2\rho_a J_a \mathcal{P}_a} 
\right),
\theta^v_2 
\left( 
1 - \frac{\theta^v_2 \rho_{\min}}{2\rho_a J_a \mathcal{P}_a} 
\right),
\theta^v_3
\left( 
1 - \frac{\theta^v_3 \rho_{\min}}{2\rho_a J_a \mathcal{P}_a} 
\right)
\right],
\end{array}
\end{equation}   
and $\Delta \vec{\bfnc{V}} =  \dover{\vec{\bfnc{V}}}-\vec{\bfnc{V}}_a$.     
 Hence, if 
 $0 \leq 
 \theta^v_l 
 \leq 2\mathcal{P}_a J_a\frac{\rho_a}{\rho_{\min}} \quad \forall \, l$ 
 then 
$T({\bf U}_a^v) \geq T({\bf U}_a)$ 
and $S({\bf U}_a^v) \leq S({\bf U}_a)$.  As follows from Eqs.~(\ref{uThetaVel}-- \ref{VelFluxLimOpt2}),
 the velocity components of $\hat{\bf U}_a^{v}$ obey:   
\begin{equation}
\label{VELCHANGE}
V_l(\hat{\bf U}_a^{v})- \dbar{V}_l 
= 
\left(
V_l(\hat{\bf U}_a)- \dbar{V}_l
\right)
\left(
1-\frac{\theta_l^v}{J_a\mathcal{P}_a}\frac{\rho_{\min}}{\rho_a}
\right) .
\end{equation}   

Since $\dover{\vec{\bfnc{V}}}$ may be changed by the limiting procedure, enforcing the velocity bound at each solution point on a given element should in principle be done iteratively, i.e., Eq.~(\ref{VELCHANGE}) can be recast in the following form:
\begin{equation}
 \label{VELIter}
 \left(V_l \right)_a^{(m)}- \dbar{V}^{(m-1)}_l 
= 
\left(
\left(V_l \right)_a^{(m-1)}- \dbar{V}^{(m-1)}_l
\right)
\left(
1-\frac{(\theta_l^v)^{(m)}}{J_a\mathcal{P}_a}\frac{\rho_{\min}}{\rho_a}
\right) ,
\end{equation}
where the superscript is the iteration number and 
$\dbar{V}^{(m)}_l 
= 
\frac{1}{N_p}\sum_{j=1}^{N_p}  \left(V_l \right)_j^{(m)}$.
Each iteration begins by finding ${\bm \theta}^v_a$ for all solution points on the element.  If the $l$-th velocity component of $\hat{\bf U}_a$ violates the velocity bound given by Eq.~(\ref{VTBOUNDS}), then we solve Eqs.~(\ref{VTBOUNDS}, \ref{VELCHANGE}) for 
$\theta_l^v$ and set 
\begin{equation}
\label{chooseTheta_i}
\left(\theta_l^v\right)_a 
= 
\mathcal{P}_a J_a\frac{\rho_a}{\rho_{\min}} 
\left(
1 - \frac{\dbar{\rho}_H h \dbar{T}_H}{\mu | \left(V_l \right)_a- \dbar{V}_l |}
\right),
\end{equation}  
otherwise we set $\left(\theta_l^v\right)_a  = 0$.  Finally,  we calculate $\theta_l^v$ as follows:
\begin{equation}
\label{chooseTheta}
\theta_l^v = \min(\max\limits_a(\left(\theta_l^v\right)_a),
\min_a(\mathcal{P}_a J_a\frac{\rho_a}{\rho_{\min}})),
\end{equation} 
alter the vector of conservative variables at each point on the element  according to Eq.~(\ref{uThetaVel}), update the velocity average, 
and repeat this iterative process until convergence. The key properties of the proposed velocity limiter are given in Theorem~\ref{velLimitTheorem}.  First, we prove the following lemma.
\begin{lemma}
\label{velLimitLemma}
At any $m$-th iteration of the method given by Eqs.~(\ref{uThetaVel}-- \ref{VelFluxLimOpt2}, \ref{VELIter}--\ref{chooseTheta}), for any $l$-th component of velocity there exist two solution points $i^{(m)}_{l,\max}$ and $i^{(m)}_{l,\min}$, such that for all $1 \le j \le N_p$ on a given element
\begin{equation}
\label{ineq1} 
\left(V_l \right)_{i^{(m)}_{l,\min}}^{(m)}
\leq  
\left(V_l \right)_j^{(m)}
\leq  
\left(V_l \right)_{i^{(m)}_{l,\max}}^{(m)},
\end{equation}
\begin{equation}
\label{ineq2} 
\left(V_l \right)_{i^{(m)}_{l,\min}}^{(m-1)}
\leq 
\dbar{V}^{(m-1)}_l  
\leq   
\left(V_l \right)_{i^{(m)}_{l,\max}}^{(m-1)},
\end{equation}
where $\left(V_l \right)_{i^{(m)}_{l,\min}}^{(m-1)}$ is the velocity at solution point ${i^{(m)}_{l,\min}}$ at the $(m-1)$-th iteration.
\begin{proof}
Let us prove the existence of an $i^{(m)}_{l,\max}$ satisfying both inequalities.  Let 
$\left(V_l \right)_a^{(m)} = \max\limits_{1\leq j\leq N_p}\left(V_l \right)_j^{(m)}$, 
so that the index ``$a$" plays the role of $i^{(m)}_{l,\max}$ in \eqref{ineq1}.  If $a$ also satisfies \eqref{ineq2}, then we can set $a=i^{(m)}_{l,\max}$ and hence such $i^{(m)}_{l,\max}$ exists.  Suppose that Eq.~(\ref{ineq2}) does not hold, i.e., $\left(V_l \right)_a^{(m-1)} < \dbar{V}^{(m-1)}_l$.  Then, there exists at least on solution point $b$ such that  $\left(V_l \right)_b^{(m-1)} > \dbar{V}^{(m-1)}_l$ and from Eqs.~\eqref{VELIter} and \eqref{chooseTheta} it follows that  $\left(V_l \right)_b^{(m)} \geq \dbar{V}^{(m-1)}_l$.  Note that it is impossible to have  $\left(V_l \right)_a^{(m)} >  \left(V_l \right)_b^{(m)} \geq \dbar{V}^{(m-1)}_l$, because $ \left(V_l \right)_a^{(m-1)} > \dbar{V}^{(m-1)}_l$ as follows from Eq.~\eqref{VELIter}.  Thus, 
$
\left(V_l \right)_a^{(m)} = \left(V_l \right)_b^{(m)},
$
so that the $b$-th solution point satisfies both Eqs.~\eqref{ineq1} and \eqref{ineq2}.
Hence, we can set $i^{(m)}_{l,\max}=b$, which again implies that such $i^{(m)}_{l,\max}$ satisfying Eqs.~\eqref{ineq1} and \eqref{ineq2} exists.  An identical argument holds for $i^{(m)}_{l,\min}$.
\end{proof}
\end{lemma}
\begin{theorem}
\label{velLimitTheorem}
The iterative method given by Eqs.~(\ref{uThetaVel}-- \ref{VelFluxLimOpt2}, \ref{VELIter}--\ref{chooseTheta}) is conservative and pointwise entropy dissipative.  
Also, the maximum possible velocity variation after $m$ iterations is bounded as follows:
\begin{equation}
\label{velBoundnthIteration}
\begin{array}{ll}
\max_a(\left(V_l \right)_a^{(m)})-\min_a(\left(V_l \right)_a^{(m)})
&
\leq
(\max_a(\left(V_l \right)_a^{(0)})-\min_a(\left(V_l \right)_a^{(0)}))
\\
&
\prod_{n=1}^m\left(1-\frac{(\theta^v_l)^{(n)}}{\max\limits_a(\mathcal{P}_a J_a\frac{\rho_a}{\rho_{\min}})}\right)  \quad \forall l.
\end{array} 
\end{equation}
Furthermore, this iterative method converges, so that upon convergence, the velocity components at all solution points satisfy the bound given by Eq.~(\ref{VTBOUNDS}).  
\end{theorem}
\begin{proof}
At each iteration, $\theta^v_l$ is computed using Eq.~(\ref{chooseTheta}), so that
$\theta^v_l \leq  \text{min}_a(\mathcal{P}_a J_a\frac{\rho_a}{\rho_{\min}})$ and
the temperature at each solution point may only increase as follows from Eq.~(\ref{tempChangeFnThetaVELSQUEEZE}).  Since  the density at each solution point remains unchanged during this limiting procedure, the mathematical entropy can only decrease.  Therefore, this iterative method is pointwise entropy dissipative.
Conservation follows from the fact that at each iteration $(\theta^v_l)^{(m)}$ is a constant on each high-order element and 
$\sum_{a=1}^{N_p} \mathcal{P}_a(\frac{1}{\mathcal{P}_a}{\bm f}_v({\bf U}_a,{\bm \theta}^v)) = 0$.

Convergence follows from the fact the iteration given  by Eq.~((\ref{uThetaVel}-- \ref{VelFluxLimOpt2}, \ref{VELIter}--\ref{chooseTheta})) is contractive.
Indeed, let $i_{l,\min}^{(m)}$ and $i_{l,\max}^{(m)}$ be defined as in Lemma~\ref{velLimitLemma}.         
Using  $(\theta^v)_l^{(m)} \leq  \min\limits_a(\mathcal{P}_a J_a\frac{\rho_a}{\rho_{\min}})$ $\forall m$ and Eq.~(\ref{VELCHANGE}), we have 
\begin{equation}
\begin{array}{ll}
\label{velContraction}
&\max_a(\left(V_l \right)_a^{(m)})-\min_a(\left(V_l \right)_a^{(m)})
 \\
&= 
\left(V_l \right)_{i_{l,\max}^{(m)}}^{(m)} - \dbar{V}^{(m-1)}_l  
+ 
\dbar{V}^{(m-1)}_l  - \left(V_l \right)_{i_{l,\min}^{(m)}}^{(m)}  \nonumber \\
&=
(\left(V_l \right)_{i_{l,\max}^{(m)}}^{(m-1)} - \dbar{V}^{(m-1)}_l  )
\left(
1-\frac{(\theta_l^v)^{(m)}}{\mathcal{P}_{i^{(m)}_{l,\max}}J_{i^{(m)}_{l,\max}}}\frac{\rho_{\min}}{\rho_{i^{(m)}_{l,\max}}}
\right)  
\nonumber 
\\&
+
(\dbar{V}^{(m-1)}_l -\left(V_l \right)_{i_{l,\min}^{(m)}}^{(m-1)} )
\left(
1-\frac{(\theta_l^v)^{(m)}}{\mathcal{P}_{i^{(m)}_{l,\min}}J_{i^{(m)}_{l,\min}}}\frac{\rho_{\min}}{\rho_{i^{(m)}_{l,\min}}}
\right)  
\nonumber 
\\
&
\leq 
(\left(V_l \right)_{i_{l,\max}^{(m)}}^{(m-1)} - \left(V_l \right)_{i_{l,\min}^{(m)}}^{(m-1)})
\left(1-\frac{(\theta_l^v)^{(m)}}{\max\limits_a(\mathcal{P}_a J_a\frac{\rho_a}{\rho_{\min}})}\right)  
\nonumber \\
&\leq
\left(
\max_a(\left(V_l \right)_a^{(m-1)})-\min_a(\left(V_l \right)_a^{(m-1)})
\right)
\left(1-\frac{(\theta_l^v)^{(m)}}{\max\limits_a(\mathcal{P}_a J_a\frac{\rho_a}{\rho_{\min}})}\right)  
\nonumber \\
&\leq
\left(
\max_a(\left(V_l \right)_a^{(0)})-\min_a(\left(V_l \right)_a^{(0)})
\right)
\prod_{n=1}^m
\left(1-\frac{(\theta_l^v)^{(n)}}{\max\limits_a(\mathcal{P}_a J_a\frac{\rho_a}{\rho_{\min}})}\right)  . \nonumber
\end{array}
\end{equation}
\end{proof}
\begin{remark}
Note that for all test problems presented in Section~\ref{results},
using only one iteration of the above iterative method per time step is sufficient to eliminate the stiffness of the time step constraint for temperature positivity for each troubled element.   
\end{remark}

\subsection{A limiter to enforce the temperature bound} 
The second step is to enforce the bound on temperature, which is given by Eq.~(\ref{VTBOUNDS}). Similar to the velocity limiter,
 we modify the temperature at each solution point by using the following limiter:
\begin{equation}  
\label{uThetaTemp} 
\hat{\bf U}_a^{t} = \hat{\bf U}_a
+ 
\frac{\theta^t}{\mathcal{P}_a}{\bm f}_t(\hat{\bf U}_a),
\, \, 
{\bm f}_t(\hat{\bf U}_a) =
\rho_{\min}
\left[ 
\begin{array}{ccccc}
0 &  0 & 0 &  0 & 
(\dbar{T} - T_a)
\end{array}
\right]^\top,
\end{equation}
where 
$\dbar{T}$ is the arithmetic average of temperature on a given high-order element.
After applying the limiter, the modified temperature is given by
\begin{equation}
\label{TEMPCHANGE}
T(\hat{\bf U}_a^{t}) - \dbar{T} 
= 
(T_a - \dbar{T})
\left(
1-\frac{\gamma-1}{R_g}\frac{\theta^t\rho_{\min}}{J_a\mathcal{P}_a\rho_a}
\right) .
\end{equation}   

If $\hat{\bf U}_a$ violates the temperature bound given by Eq.~(\ref{VTBOUNDS}), then we set 
\begin{equation}
\label{thetaiTempSqueeze}
\theta_a^t = J_a\mathcal{P}_a\frac{\rho_a}{\rho_{\min}} \frac{R_g}{\gamma-1}
\left(1 - \frac{T_a \dbar{T}}{|T_a - \dbar{T}|}\frac{\dbar{\rho}_H h}{\tilde{\lambda}_a \mu}\right),
\end{equation}  
otherwise we set $\theta^t_a = 0$. Note that by construction, $0 \le \theta^t_a \le 1$, $\forall a$.  Finally, the temperature limiter is defined as follows:
\begin{equation}
\label{chooseThetaTemp}
\theta^t = 
\min
\left(
\max\limits_a(\theta_a^t),\frac{R_g}{\gamma-1}\min_a(J_a\mathcal{P}_a\frac{\rho_a}{\rho_{\min}} )
\right)
\end{equation} 
and $\hat{\bf U}_a^{t}$ at all solution points on the element is modified according to Eq.~(\ref{uThetaTemp}).
Similar to the velocity limiter, the temperature limiting procedure should in general be performed iteratively.
The key properties of the proposed temperature limiter are presented in the following theorem.
\begin{theorem}
The iterative temperature limiting procedure given by Eqs.~(\ref{uThetaTemp}, 
\ref{thetaiTempSqueeze}, \ref{chooseThetaTemp}) 
is conservative and elementwise entropy dissipative.  
Also, the maximum possible temperature variation after $m$ iterations is bounded as follows:
\begin{equation}
\label{tempBoundnthIteration}
\begin{array}{ll}
\max\limits_a(T_a^{(m)})-\min_a(T_a^{(m)}) \leq
& (\max_a(T_a^{(0)})-\min_a(T_a^{(0)}) ) \\
& \prod_{n=1}^m\left(1-\frac{\gamma-1}{R_g}
\frac{(\theta^t)^{(n)}}
{\max_a(J_a\mathcal{P}_a\frac{\rho_a}{\rho_{\min}} )}
\right).
\end{array}
\end{equation}
Furthermore, this iterative method 
converges, so that upon convergence, the temperature at all solution points satisfies the bound given by Eq.~(\ref{VTBOUNDS}).  
\end{theorem}
\begin{proof}
Conservation follows from the fact that $(\theta^t)^{(m)}$ is a constant on each high-order element and
$\sum_{a=1}^{N_p} 
\mathcal{P}_a
(\frac{(\theta^t)^{(m)}}{\mathcal{P}_a}{\bm f}_t(\hat{\bf U}^{(m-1)}_a)) = 0$.  

Let us show that the temperature limiting procedure is elementwise entropy dissipative, i.e.,
\begin{equation}
\begin{array}{l}
\label{Sstabl}
\sum_{a=1}^{N_p} \mathcal{P}_a J_a S({\bf U}^t_a) 
\leq
\sum_{a=1}^{N_p} \mathcal{P}_a J_a S({\bf U}_a).
\end{array}
\end{equation}
Note that it is sufficient to show that the entropy dissipates at the first iteration, because the same argument holds for all other iterations as well.  
Let $I_L$, $I_E$ and $I_G$ be the following index sets: $T_a < \dbar{T} \ \forall a \in I_L$, $T_a = \dbar{T} \ \forall a \in I_E$, and $T_a > \dbar{T} \ \forall a \in I_G$, respectively.
Taking into account that 
\begin{equation}
\label{dSdt_tempSQUEEZE}
\frac{dS({\bf U}_a(\theta^t))}{d\theta^t} 
= 
-\frac{\rho_{\min}}{J_a\mathcal{P}_a}
\frac{\dbar{T}-T_a}{T({\bf U}_a(\theta^t))},
\end{equation}
 we have
 \begin{equation}
 \begin{array}{ll}
\frac{d}{d\theta^t}
\sum_{a=1}^{N_p}
 \mathcal{P}_a J_a S({\bf U}_a(\theta^t))
&
= 
\sum_{a=1}^{N_p} 
\mathcal{P}_a J_a  
\frac{dS({\bf U}_a(\theta^t))}{d\theta^t} 
=
-\rho_{\min}
\sum_{a=1}^{N_p}\frac{\dbar{T} - T_a}{T({\bf U}_a(\theta^t))}
\nonumber
\\
&=
\rho_{\min} 
\left( 
-\sum_{a \in I_L}\frac{\dbar{T} - T_a}{T({\bf U}_a(\theta^t))}
+\sum_{a \in I_G}\frac{T_a - \dbar{T}}{T({\bf U}_a(\theta^t))}
\right)
\nonumber \\
&\leq
\rho_{\min} 
\left( 
-\sum_{a \in I_L}\frac{\dbar{T} - T_a}{\dbar{T}}
+\sum_{a \in I_G}\frac{T_a - \dbar{T}}{\dbar{T}}
\right)
\nonumber \\
&=
\frac{\rho_{\min} }{\dbar{T}} \left(-\sum_{a \in I_L}\dbar{T} - T_a
+\sum_{a \in I_G}T_a - \dbar{T}\right)
\nonumber \\
&=
\frac{\rho_{\min} }{\dbar{T}}
\left(-N_p \dbar{T} + \sum_{a=1}^{N_p}T_a\right) = 0,
\end{array}
\end{equation}
so long as 
$0 
\leq 
\theta^t 
\leq 
\frac{R_g}{\gamma-1}
\min\limits_a(J_a\mathcal{P}_a\frac{\rho_a}{\rho_{\min}} )$ 
which is the case when $\theta^t$ is selected according to Eq.~(\ref{chooseThetaTemp}).  Thus, 
$\sum_{a=1}^{N_p} \mathcal{P}_a J_a S({\bf U}_a(\theta^t))$
is non-increasing as a function of $\theta^t$ on 
$\leq 
\theta^t 
\leq 
\frac{R_g}{\gamma-1}
\min\limits_a(J_a\mathcal{P}_a\frac{\rho_a}{\rho_{\min}} )$
 and satisfies Eq.~(\ref{Sstabl}).  The proof of the temperature bound given by Eq.~(\ref{tempBoundnthIteration}) relies on 
Eqs.~(\ref{TEMPCHANGE}, \ref{chooseThetaTemp}) and is nearly identical to the proof of the velocity bound (Eq.~(\ref{velBoundnthIteration})) and therefore not presented herein.  
Together Eq.~\eqref{tempBoundnthIteration} and Eq.~\eqref{chooseThetaTemp} imply that the temperature variation decreases with each iteration.  Hence, the bound in Eq.~\eqref{VTBOUNDS}  
is met after a finite number of iterations, because 
$\min\limits_a(T_a) \leq T_a^{(m)} \leq \max\limits_a(T_a)  \ \forall m$ and the following lower bound holds: 
\begin{align}
\frac{T_a^{(m)} \dbar{T}^{(m)}\dbar{\rho}_H h }
{\tilde{\lambda}^{(m)}_a \mu} 
\geq
\frac{(\min\limits_a(T_a))^2\dbar{\rho}_H h }{\left[ \max\limits_a(\|\bfnc{V}\|_a) + c(\max\limits_a(T_a)) \right] \mu} 
> 0,
\end{align}
where
$\dbar{T}^{(m)}$ is the arithmetic average of temperature on the element. 
\end{proof}
It should be emphasized again that only one iteration per each troubled element per time step is sufficient to eliminate the stiffness of the temperature positivity time constraint for all test problems considered.

\subsection{Consistency of the velocity and temperature limiting procedure}
\label{OrderOfAccuracyVISLIM}
Form Eqs.~(\ref{VelFluxLimOpt2}) and (\ref{uThetaTemp}), it follows that  the velocity and temperature limiters are first-order accurate, i.e.,
$\|\hat{\bf U}_a^{v} - \hat{\bf U}_a\| =O(h)$ and $\|\hat{\bf U}_a^{t} - \hat{\bf U}_a\| =O(h)$.
Hence, the above limiting procedure is design-order accurate.

\section{Numerical Results}
\label{results}
To assess accuracy, discontinuity-capturing, and positivity preservation properties of the proposed first-order entropy stable FV scheme for the 3-D compressible Navier-Stokes equations, we consider standard benchmark problems with smooth and discontinuous solutions. 
In all numerical experiments  presented herein, the first-order explicit Euler method is used to advance the semi-discretization in time. Note
that this scheme violates the entropy stability property of the semi-discrete
operator by a factor proportional to the local temporal truncation error. 
The time step in all numerical experiments is selected by using the Courant-Friedrich-Levy (CFL)-type condition 
and the density and temperature positivity constraints presented in Section~\ref{POSPRE}.


\subsection{$3$-D Viscous Shock}
%
%
The first test problem is the propagation of a $3$-D viscous shock on non-uniform randomly perturbed grids.
Since this problem possesses a smooth analytical solution \cite{Fthesis}, we use it to verify the order property of the proposed scheme.
Note, however, that on coarse meshes, the viscous shock is under-resolved and behaves as a strong discontinuity. 
\begin{table}[!h]
\begin{center}
\begin{tabular}{ccccc}
\hline
$K$   & $L_{\infty}$ error  &  rate  & $L_2$ error  &  rate \\
\hline
  3  & 1.21          &  --         & 1.02e-1    & --    \\
  6  & 7.74e-1      &  0.65     & 6.74e-2    & 0.60  \\
 12  & 9.57e-1     & -0.31     & 4.10e-2    & 0.72  \\
 24  & 5.46e-1     &  0.81     & 2.39e-2    & 0.79  \\
 48  & 3.02e-1     &  0.85     & 1.25e-2    & 0.94  \\
 96  & 1.42e-1     &  1.09     & 6.39e-3    & 0.97  \\
\hline
\end{tabular}
\end{center}
\caption{\label{tab2} $L_{\infty}$ and $L_2$ errors and their convergence rates obtained 
on randomly perturbed grids with $p = 4$ LGL elements for the 3-D viscous shock problem at $Ma=2.5, Re=50$. }
\end{table}
To make the problem three-dimensional, a planar viscous shock, which is initially centered at the origin, is rotated, so that it propagates 
along the direction $[1, 1, 1]^\top$.
The flow parameters are set as follows: $Re = 50$, $Ma = 2.5$, and $Pr = 3/4$, and the simulation is run until $t_{\text{final}} = 0.1$.   
The randomly perturbed non-uniform grids are constructed from corresponding uniform grids with $K^3$ total elements partitioning the $-0.5 \leq x,y,z \leq 0.5$  domain by adding $r/K$ to each coordinate of each vertex in the domain, where the variable $r$ is a random number such that $0 \leq r < 0.4$.  As follows from the results presented in Table~\ref{tab2}, the $L_{\infty}$ and $L_2$ errors and their convergence rates obtain on these grids corroborate that 
 the proposed FV scheme is first-order accurate.
\begin{figure}[!h] 
	\begin{subfigure}{0.5\textwidth}
		\includegraphics[width=0.9\linewidth]{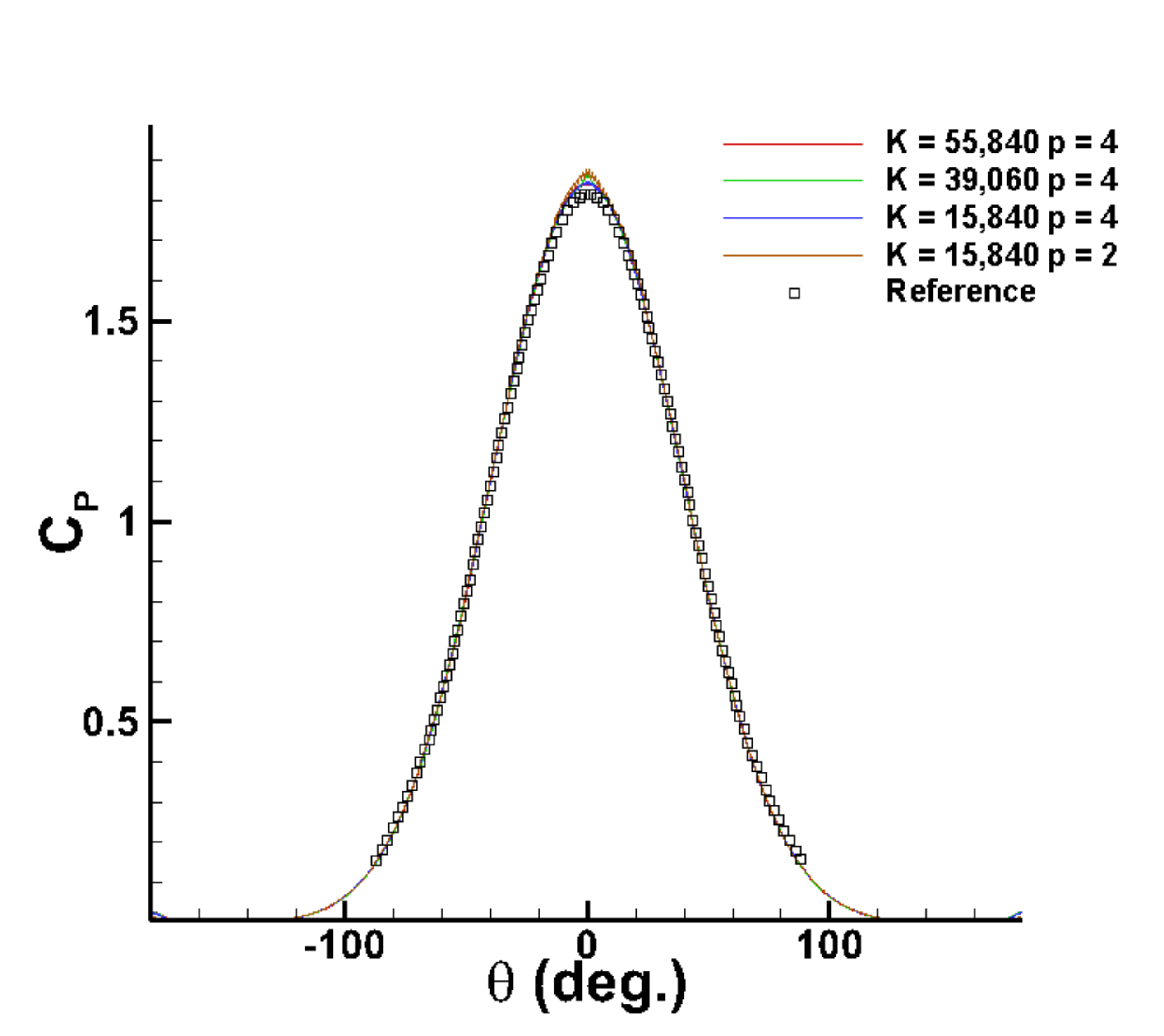} 
		\caption{}
	\end{subfigure}
	\begin{subfigure}{0.5\textwidth}
		\includegraphics[width=0.9\linewidth]{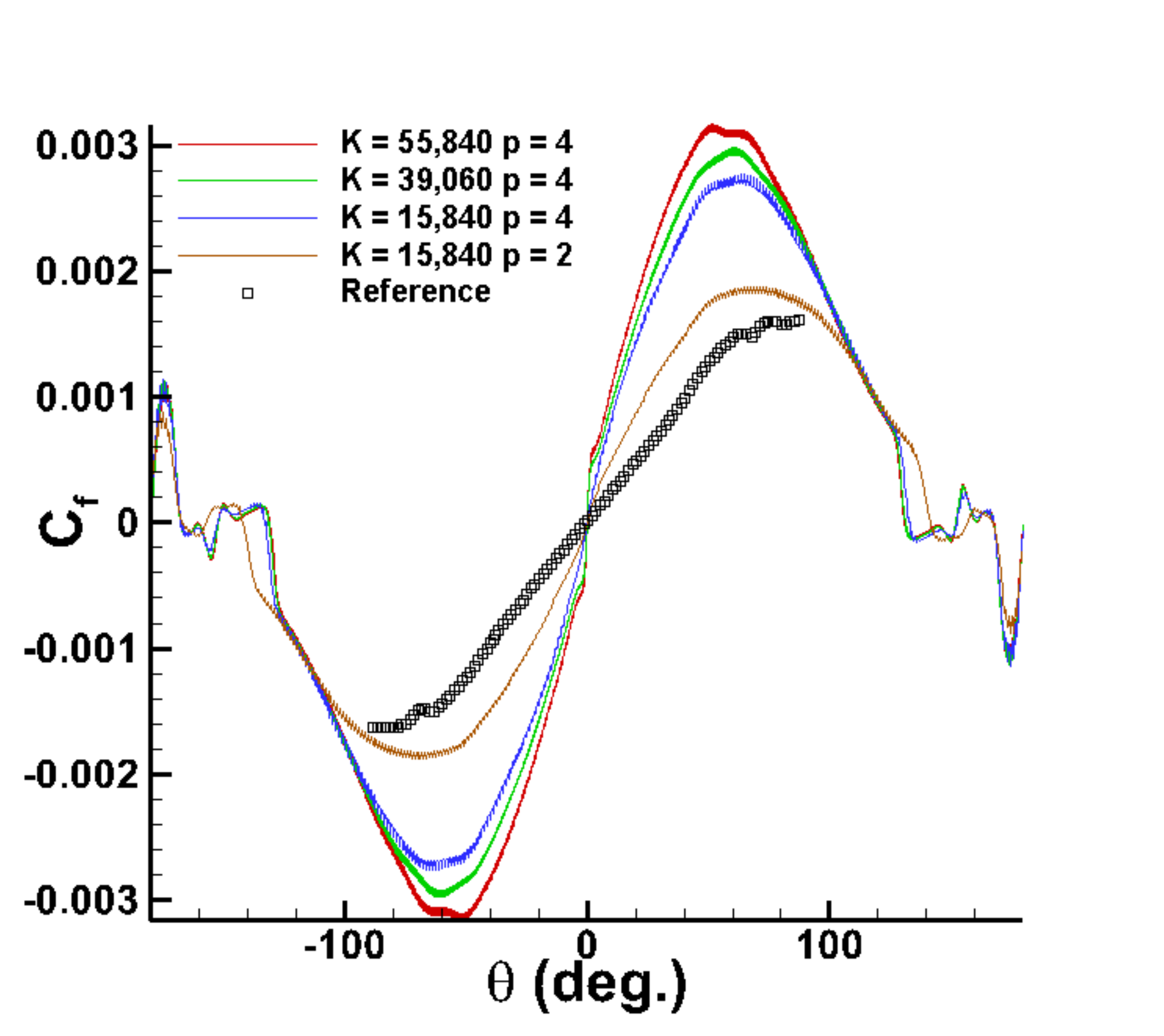}
		\caption{}
	\end{subfigure}
	 \caption{ Time-averaged wall pressure (left) and skin friction (right) coefficients obtained with the present 1st-order positivity-preserving entropy stable scheme and fourth-order HDG method \cite{fernandez2018physics} for the hypersonic cylinder flow. }
\label{skin_friction_pressure_Hcylinder}
\end{figure}

\subsection{$2$-D hypersonic cylinder}
%
%
The next test problem is the hypersonic flow around a two-dimensional adiabatic cylinder of diameter 1. The flow parameters are the same used in \cite{fernandez2018physics}:  $Re = 376,930$, $Ma = 17.605$, and $Pr=0.71$.     
Initially, the flow is uniform with $\rho = 1$, $T = 1$, and 
$\vec{\bfnc{V}} = \left[1, 0, 0\right]^\top$.    
The discretely entropy stable, adiabatic no-slip wall boundary conditions developed in \citep{SOLIDWALL2019} are used at the cylinder surface.
Three grids, coarse ($15,840$ elements), medium ($39,060$ elements), and fine ($55,260$ elements ), are considered for this problem.
The grid is stretched in the radial direction and its  resolution near the cylinder wall is $\Delta r = 2.67\cdot 10^{-3}$,  $\Delta r = 2\cdot 10^{-3}$ and $\Delta r = 1.33\cdot 10^{-3}$ for the coarse, medium, and fine grids, respectively.  The grid is uniform in the circumferential direction and has 288, 560, and 720 radial lines for the coarse, medium, and fine grids, respectively.  
%
%
\begin{figure}[!h] 
	\begin{subfigure}{0.5\textwidth}
		\includegraphics[width=0.9\linewidth]{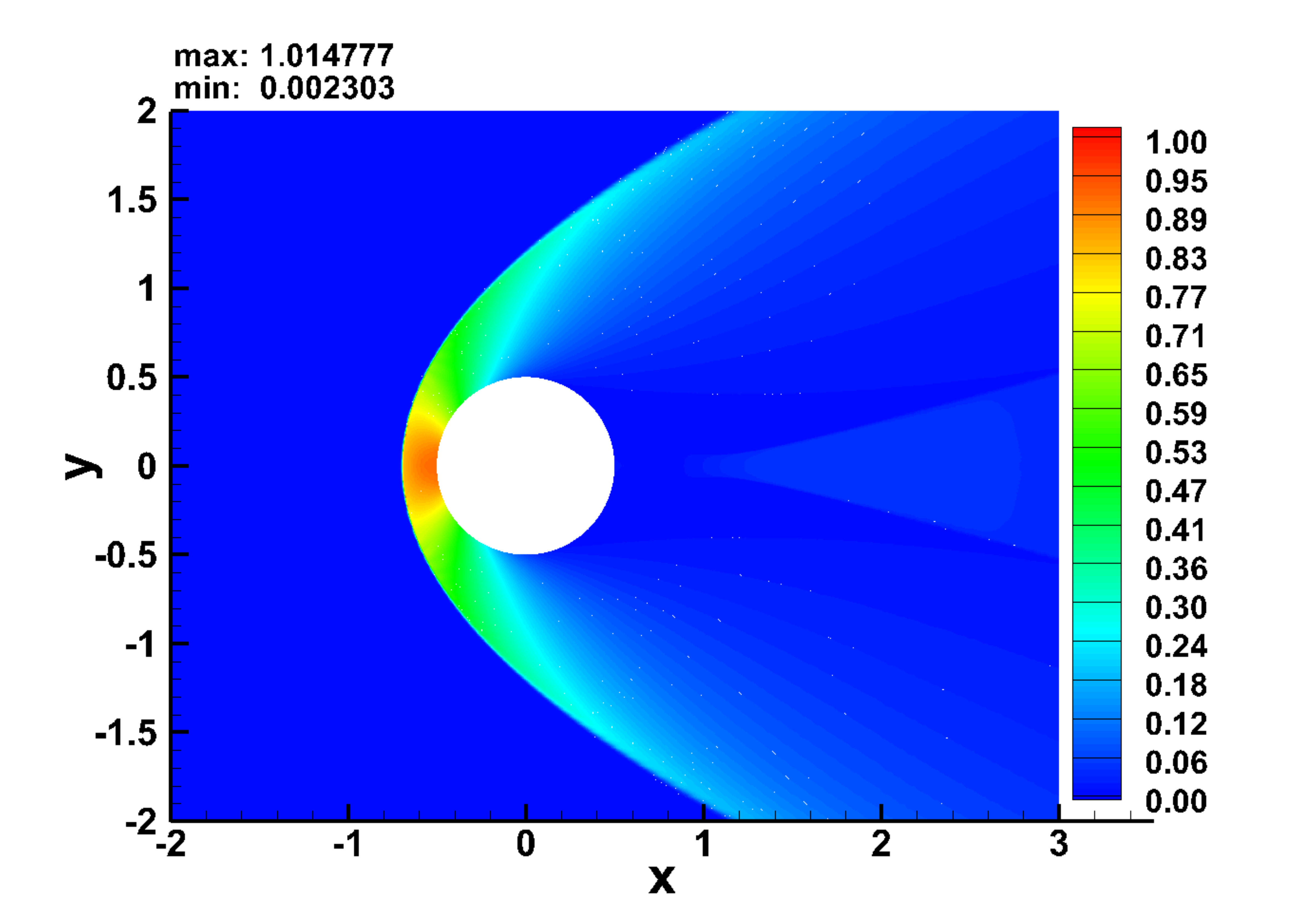} 
		\caption{}
	\end{subfigure}
	\begin{subfigure}{0.5\textwidth}
		\includegraphics[width=0.9\linewidth]{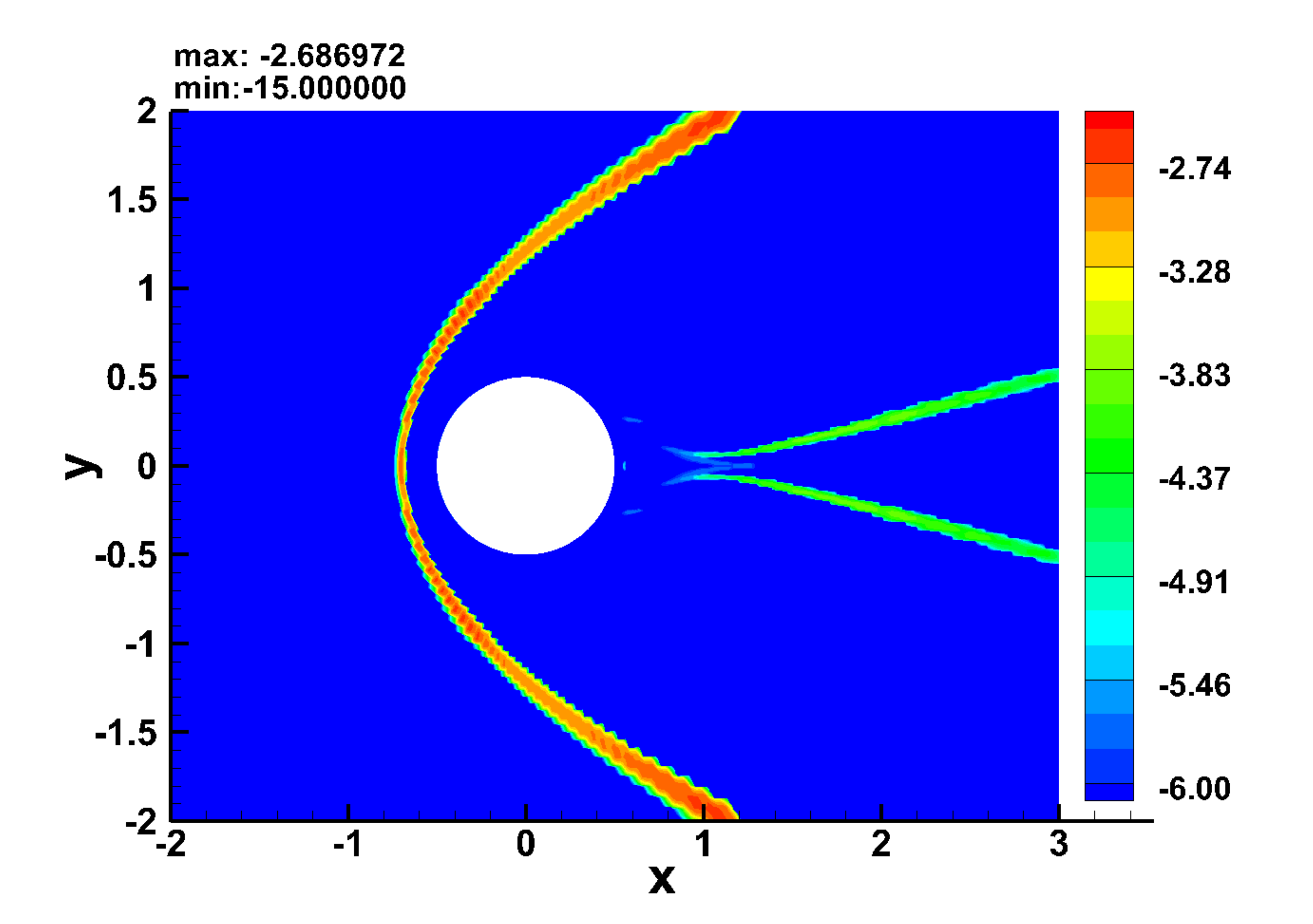}
		\caption{}
	\end{subfigure}
	 \caption{Density (left) and artificial viscosity ($\log_{10}$) are shown for the $p=4$ fine grid solution of the hypersonic cylinder problem at $t=20$. }
\label{densPressVortMach_Hcylinder}
\end{figure} 
Figure~\ref{skin_friction_pressure_Hcylinder} shows the comparison of time-averaged wall pressure and skin-fiction coefficients obtained with the new first-order  positivity-preserving entropy stable scheme and  the fourth-order hybridized discontinuous Galerkin (HDG) method developed in \cite{fernandez2018physics}.
Note that the grid used in \cite{fernandez2018physics} has only $16,000$ elements, which is comparable with the coarse grid used in the present analysis.  
The pressure coefficients computed with both schemes agree very well for all grids considered. 
Note, however, that the agreement between the skin-friction coefficients is less satisfactory.
Based on our numerical results, the skin friction increases as the grid is refined, as one can see in Fig.~\ref{skin_friction_pressure_Hcylinder}.
This discrepancy in the skin-friction coefficients can be attributed to the lack of the grid resolution in the boundary layer provided by the $16,000$-element grid used in \cite{fernandez2018physics}.
Among other possible factors that may have a negative impact on the accuracy of the skin friction coefficient are excessive artificial dissipation added in the boundary layer region or a significantly different time-averaging window used in \cite{fernandez2018physics}.

\subsection{$2$-D shock diffraction}
The last test problem is the diffraction of a rightward moving shock of Mach number $200$ over a backward facing corner for both inviscid and viscous flow regimes.  It is well known that numerical schemes are prone to producing negative densities and pressures in simulation of strong shocks diffracting over sharp corners; hence, this test problem is very well suited for evaluating the robustness of the proposed scheme.
\begin{figure}[!h] 
	\begin{subfigure}{0.5\textwidth}
		\includegraphics[width=0.9\linewidth]{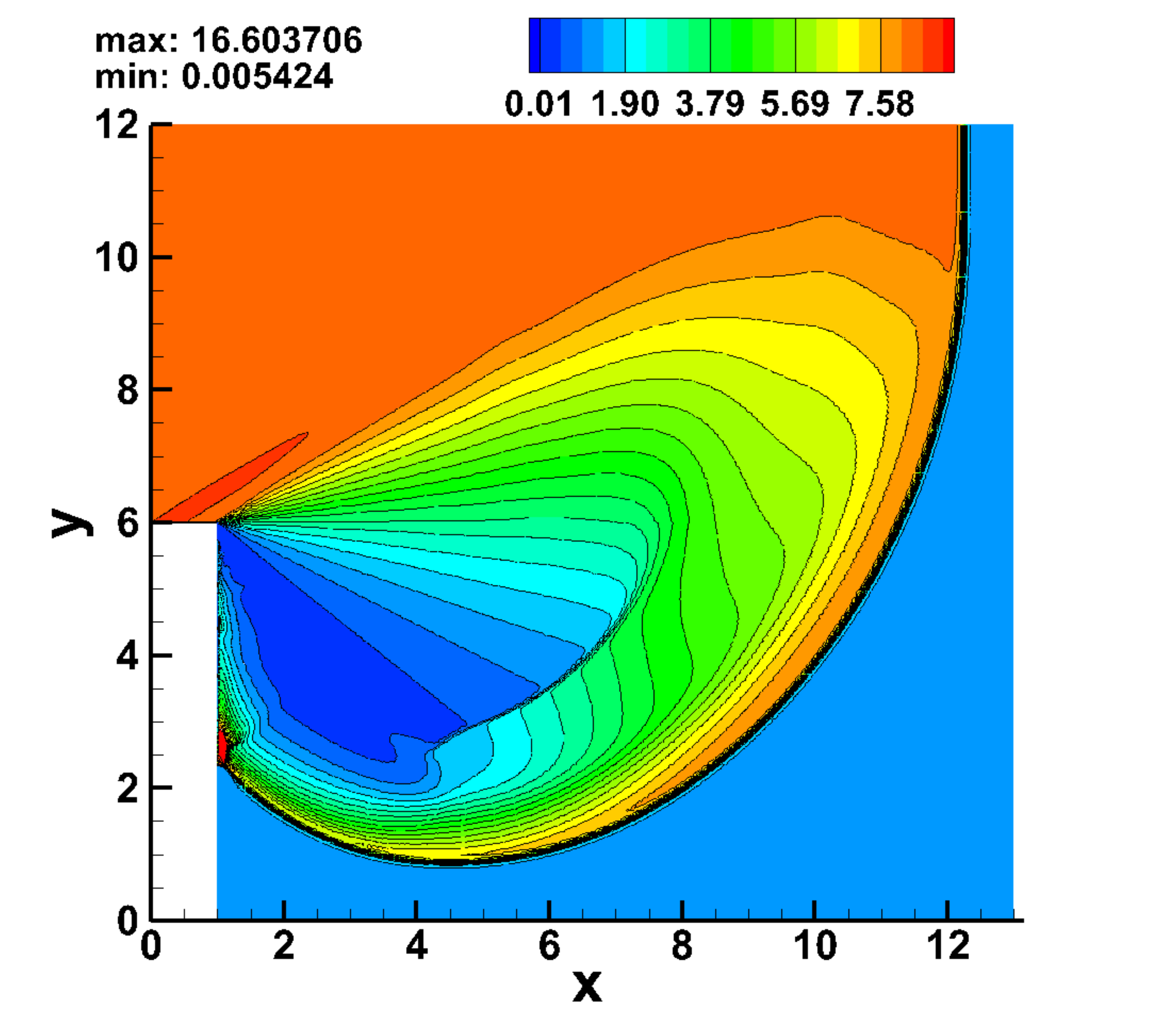} 
		\caption{}
	\end{subfigure}
	\begin{subfigure}{0.5\textwidth}
		\includegraphics[width=0.9\linewidth]{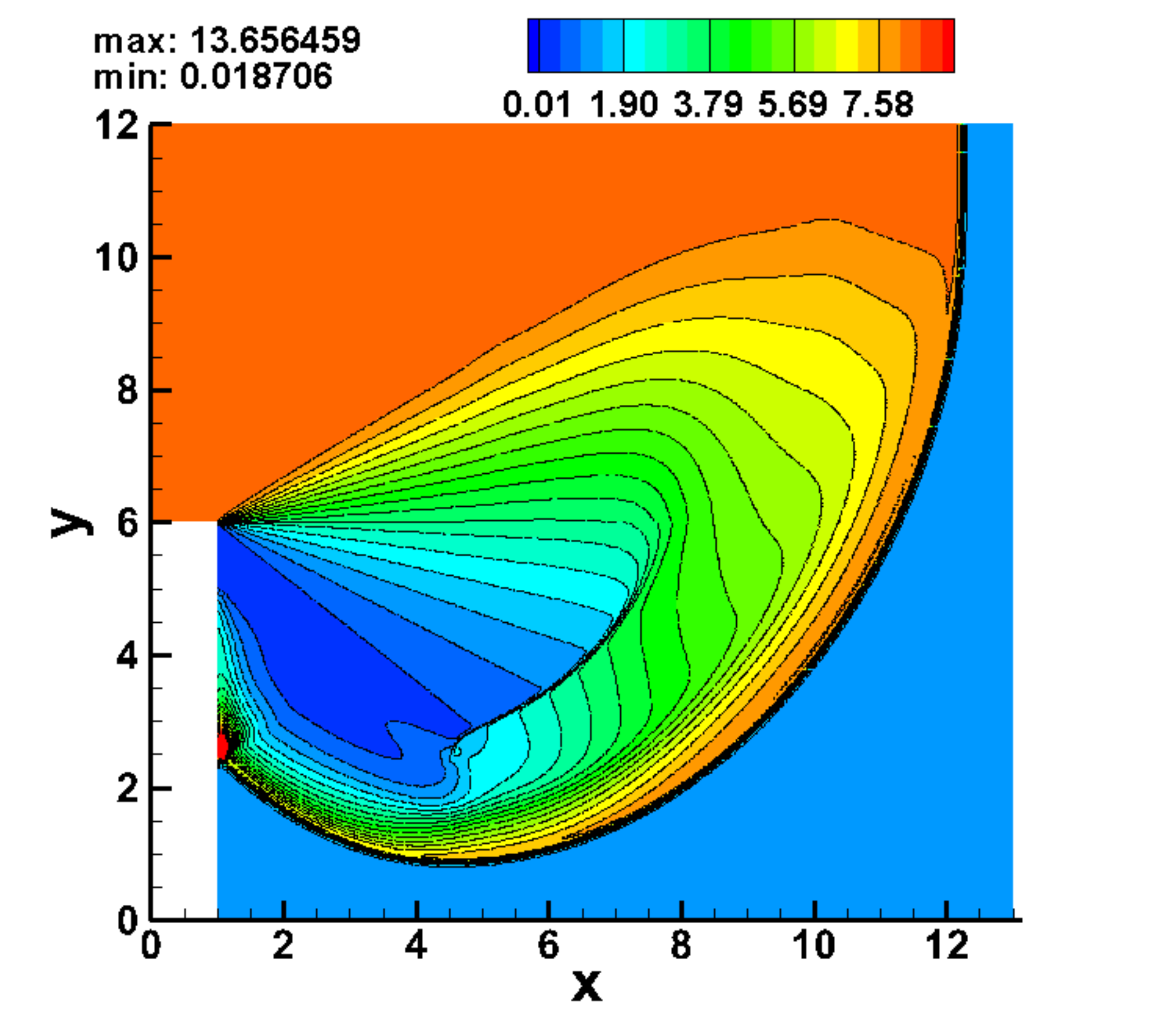}
		\caption{}
	\end{subfigure}
	\\
	\begin{subfigure}{0.5\textwidth}
		\includegraphics[width=0.9\linewidth]{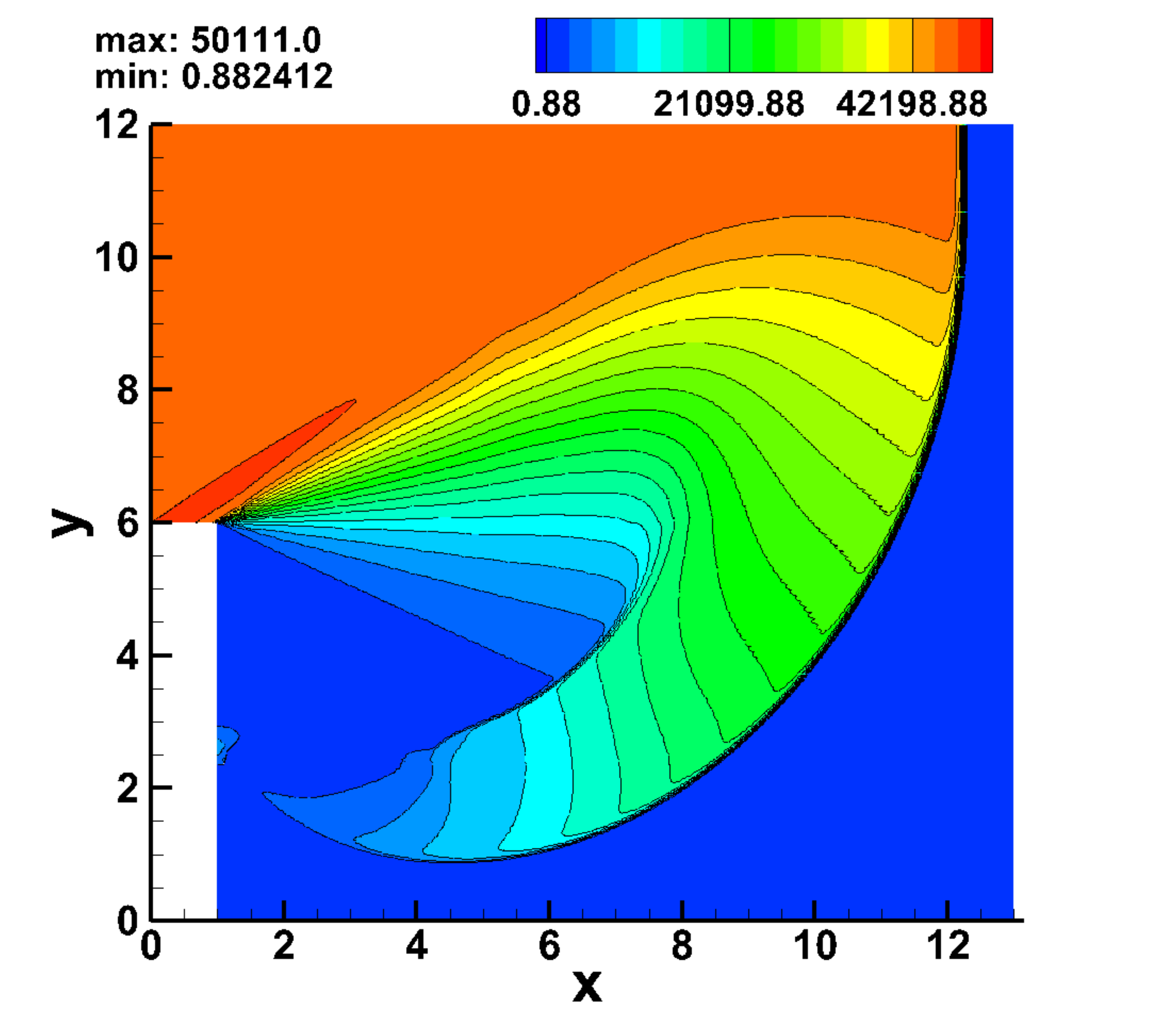} 
		\caption{}
	\end{subfigure}
	\begin{subfigure}{0.5\textwidth}
		\includegraphics[width=0.9\linewidth]{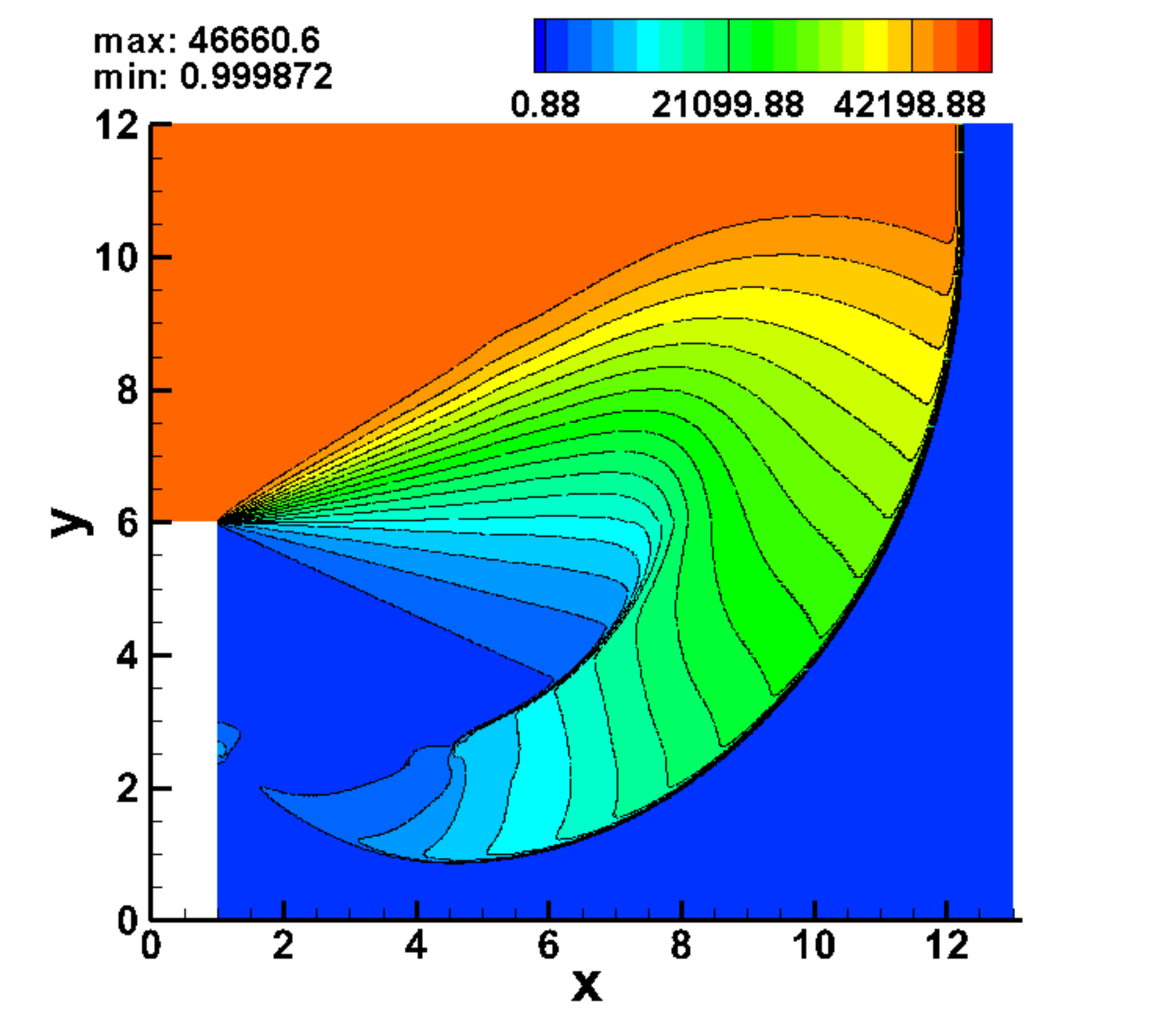}
		\caption{}
	\end{subfigure}
	 \caption{Density (top row) and  pressure (bottom row) contours obtained with the first-order positivity-preserving scheme for the viscous (left) and inviscid (right) shock diffraction flows at $Ma=200$.} 
\label{densityPres_shockDiffraction_M200}
\end{figure}

%
%
\begin{figure}[!h] 
	\begin{subfigure}{0.5\textwidth}
		\includegraphics[width=0.9\linewidth]{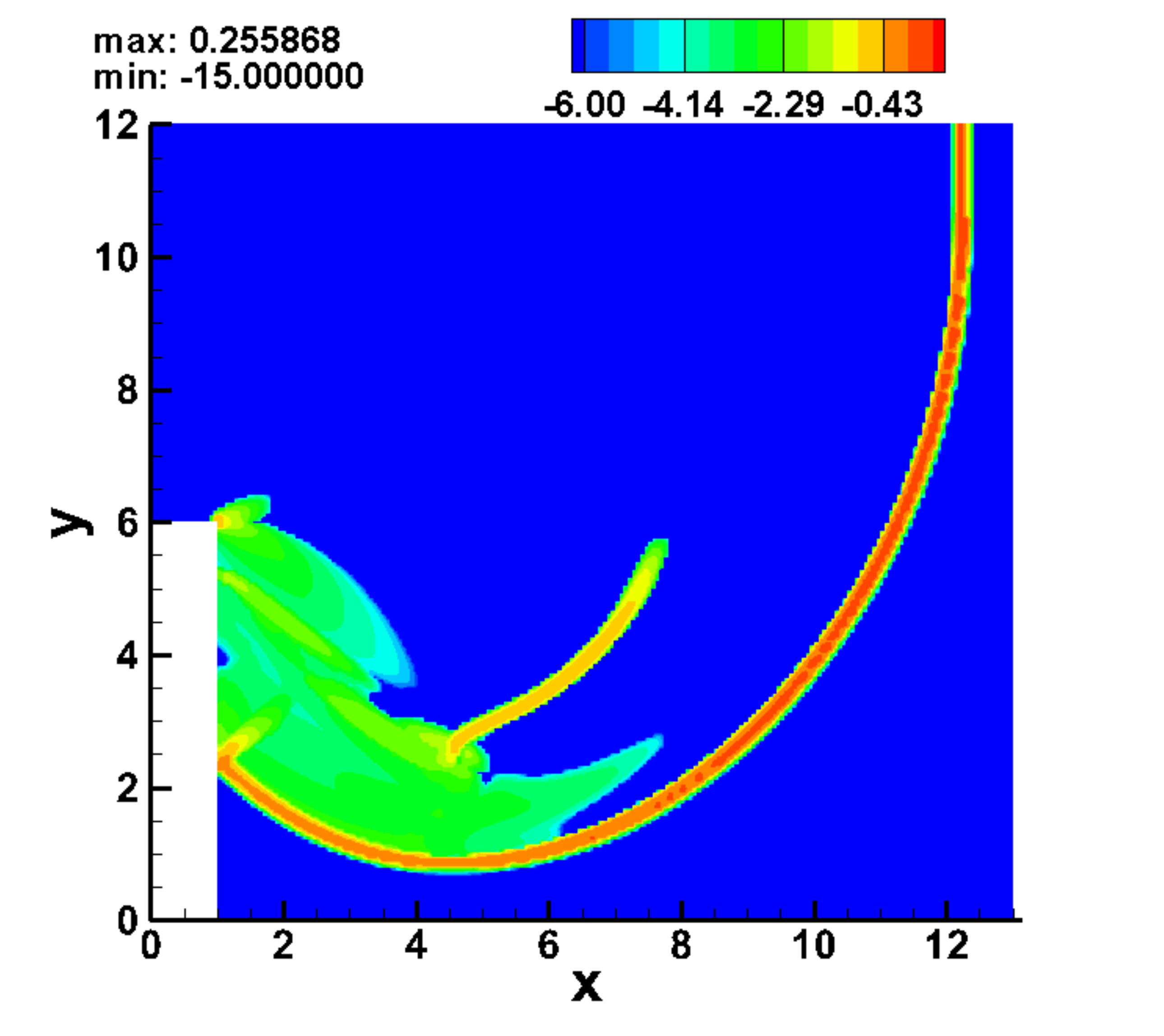} 
		\caption{}
	\end{subfigure}
	\begin{subfigure}{0.5\textwidth}
		\includegraphics[width=0.9\linewidth]{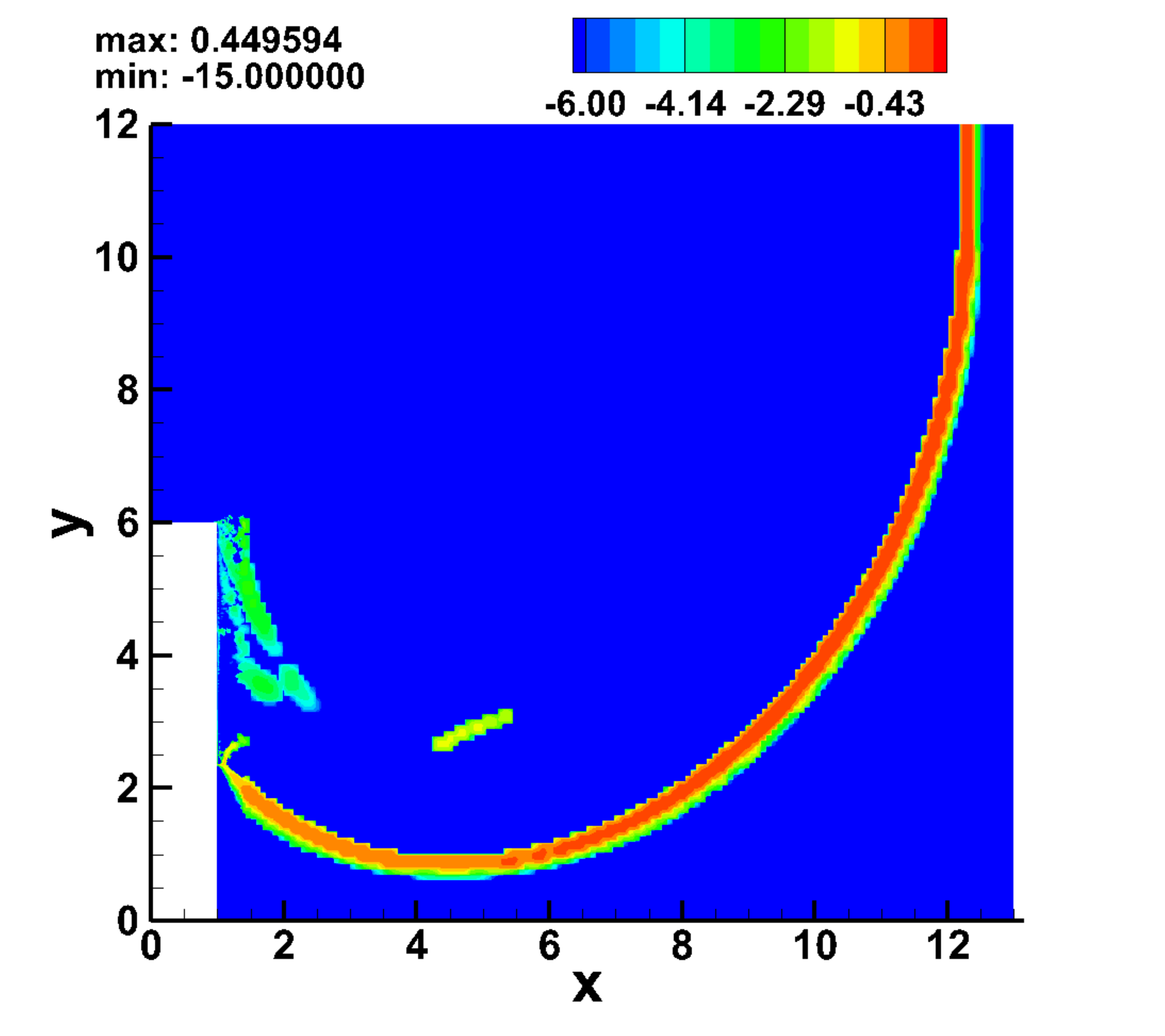}
		\caption{}
	\end{subfigure}
	 \caption{Artificial viscosity ($\log_{10}$) of the inviscid (left) and viscous (right) shock diffraction problem.  } 
\label{logmuad_shockDiffraction_M200}
\end{figure}

In contrast to the results presented in \cite{Zhang}, we use the entropy stable adiabatic no-slip boundary conditions at the corner wall \citep{SOLIDWALL2019}.
Initially, the rightward moving shock of Mach number 200 is located at $x = 0.5$.  On the downstream side of the shock, the initial conditions are   
$\rho = 1.4$, $\fnc{P} = 1$, and $\vec{\bfnc{V}} = 0$.  The solution upstream of the shock is determined by using the Rankine–Hugoniot conditions and the given shock speed.  For the viscous flow, we use the Blasius boundary layer solution on the upstream side of the shock with the freestream conditions corresponding to the Mach number of 200.  The governing equations are integrated until $t=5.8535\cdot 10^{-2}$.  For the viscous flow, the Sutherland's law is used and the Prandtl and Reynolds numbers are set equal to $0.75$ and $10^4$, respectively.
%
%
A uniform rectangular mesh with constant grid spacings $\Delta x = \Delta y$ and $40,000$ elements is used for the inviscid flow simulation. 
For the viscous flow case, the grid consists of $52,944$ elements and is clustered near the corner surface,  so that the normal grid spacing at the wall is $2.67\times 10^{-3}$.

Density, pressure, and artificial viscosity contours for the inviscid and viscous flows at the final time are presented in Figures~\ref{densityPres_shockDiffraction_M200} and \ref{logmuad_shockDiffraction_M200}.  
As follows from these results, the new first-order positivity-preserving entropy stable scheme captures  both the weak and strong shocks as well as the contact discontinuity within one grid element practically without producing spurious oscillations.
It should also be noted that the artificial viscosity coefficient is about 2 orders of magnitude smaller at the contact discontinuity than at the shock, thus indicating that the proposed physics-based artificial dissipation method is capable of distinguishing shocks from contact discontinuities.  

\bigskip  
{\bf Acknowledgments}
The first author was supported by the Virginia Space Grant Consortium
Graduate STEM Research Fellowship and the Science, Mathematics
and Research for Transformation (SMART) Scholarship. The second author
acknowledges the support from Army Research Office through grant
W911NF-17-0443.


\end{document}